\documentclass[11pt]{amsart}

\usepackage{amsmath,amssymb,amsthm,amsfonts,amscd,flafter,epsf,epsfig,graphicx,verbatim,pinlabel, mathrsfs}
\usepackage[all]{xy}

\title[On the spectral sequence from $\kh(L)$ to $\hf(-\Sigma(L))$]{On the spectral sequence from Khovanov homology to Heegaard Floer homology}

\author[John A.\ Baldwin]{John A.\ Baldwin}
\address{Department of Mathematics,
Princeton University,
Princeton, NJ 08544-1000, USA}
\email{baldwinj@math.princeton.edu}
\urladdr{http://math.princeton.edu/\char126 baldwinj}
\thanks{The author was partially supported by an NSF Postdoctoral Fellowship.}

\newcommand\wt{\widetilde}

\newcommand\kh{\widetilde{Kh}}
\newcommand\ckh{\widetilde{CKh}}

\newcommand\hf{\widehat{HF}}

\newcommand\cf{\widehat{CF}}

\newcommand\zzt{\mathbb{Z}_2}

\newcommand\T{\mathbb{T}}

\newcommand\Mh{\mathcal{M}}

\newcommand\khc{\psi}

\newcommand\ol{\widehat}
\newcommand\bv{\mathbf{v}}
\newcommand\bx{\mathbf{x}}
\newcommand\bz{\mathbf{z}}
\newcommand\by{\mathbf{y}}
\newcommand\wM{\widetilde{\mathcal{M}}}
\newcommand\Pn{\mathbf{P}_n}
\newcommand\bP{\mathbf{P}}
\newcommand\Co{\mathcal{C}}
\newcommand\hCo{\widehat{\mathcal{C}}}
\newcommand\bL{\mathbb{L}}
\newcommand\wh{\widehat}
\newcommand\htheta{\wh\Theta}

\newtheorem{theorem}{Theorem}[section]
\newtheorem{lemma}[theorem]{Lemma}

\newtheorem{conjecture}[theorem]{Conjecture}
\newtheorem{corollary}[theorem]{Corollary}
\newtheorem{proposition}[theorem]{Proposition}

\theoremstyle{definition}
\newtheorem{definition}[theorem]{Definition}

\newtheorem{remark}[theorem]{Remark}

\begin{document}
\begin{abstract}  
Ozsv{\'a}th and Szab{\'o} show in \cite{osz12} that there is a spectral sequence whose $E^2$ term is $\kh(L),$ and which converges to $\hf(-\Sigma(L))$. We prove that the $E^k$ term of this spectral sequence is an invariant of the link $L$ for all $k\geq 2$. If $L$ is a transverse link in $(S^3,\xi_{std})$, then we show that Plamenevskaya's transverse invariant $\psi(L)$ gives rise to a transverse invariant, $\psi^k(L)$, in the $E^k$ term for each $k\geq 2$. We use this fact to compute each term in the spectral sequences associated to the torus knots $T(3,4)$ and $T(3,5)$.
\end{abstract} 

\maketitle
\section{Introduction}

Let $\Sigma(L)$ denote the double cover of $S^3$ branched along the link $L$. In \cite{osz12}, Ozsv{\'a}th and Szab{\'o} construct a spectral sequence whose $E^2$ term is the reduced Khovanov homology $\kh(L)$, and which converges to the Heegaard Floer homology $\hf(-\Sigma(L))$ (using $\zzt$ coefficients throughout). Although the definition of $\kh(L)$ is intrinsically combinatorial and there is now a combinatorial way to compute $\hf(-\Sigma(L))$ \cite{sw}, the higher terms in this spectral sequence have remained largely mysterious.  For instance, the construction in \cite{osz12} depends \emph{a priori} on a planar diagram for $L$, and the question of whether these higher terms are actually invariants of the link $L$ has remained open since Ozsv{\'a}th and Szab{\'o} introduced their link surgeries spectral sequence machinery in 2003.

The primary goal of this paper is to show that for $k\geq 2$, the $E^k$ term in this spectral sequence is an invariant, as a graded vector space, of the link $L$; that is, it does not depend on a choice of planar diagram. This gives rise to a countable sequence of link invariants $\{E^k(L)\}$, beginning with $E^2(L) \cong \kh(L),$ and ending with $E^{\infty}(L) \cong \hf(-\Sigma(L))$. 
It is our hope that knowing that these higher terms are link invariants will inspire attempts to compute and make sense of them. In particular, it seems plausible that there is a nice combinatorial description of the higher differentials in this spectral sequence. Such a description would, among other things, lead to a new combinatorial way of computing $\hf(-\Sigma(L))$ (and perhaps $\hf(Y)$ for any 3-manifold $Y$, using the \emph{Khovanov homology of open books} construction in \cite{baldpla}). 

One of the first steps in this direction may involve understanding how the higher differentials behave with respect to the \emph{$\delta$-grading} on $\kh(L)$, which is defined to be one-half the quantum grading minus the homological grading. When $\kh(L)$ is supported in a single $\delta$-grading, the spectral sequence collapses at $E^2(L) \cong \kh(L)$. Therefore, one might conjecture that all higher differentials shift this $\delta$-grading by some non-zero quantity. Along these lines, it is natural to ask whether there is a well-defined quantum grading on each $E^k(L)$, and, if so, how the induced $\delta$-grading on $E^{\infty}(L)$ compares with the Maslov grading on $\hf(-\Sigma(L))$ or with the conjectured grading on $\hf(-\Sigma(L))$ described in \cite[Conjecture 8.1]{greene}. We propose the following.

\begin{conjecture}
\label{conj:quant}
For $k\geq 2$, there is a well-defined quantum grading (resp. $\delta$-grading) on each $E^k(L)$, and the $D^k$ differential increases this grading by $2k-2$ (resp. $-1$).  \end{conjecture}

Although the terms $E^0(L)$ and $E^1(L)$ are not invariants of the link $L$, they provide some motivation for this conjecture. Recall that $(E^1(L),D^1)$ is isomorphic to the complex for the reduced Khovanov homology of $L$ \cite{osz12}. Under this identification, the induced quantum grading on $E^1(L)$ is (up to a shift) the homological grading plus twice the intrinsic Maslov grading. If we define a quantum grading on $E^0(L)$ by the same formula, then, indeed, $D^k$ increases quantum grading by $2k-2$ for $k=0,1$. Moreover, Josh Greene observes that Conjecture \ref{conj:quant} holds for almost alternating links (and almost alternating links account for all but at most 3 of the 393 non-alternating links with 11 or fewer crossings \cite{adams,gyh}). 

If this conjecture is true in general, then we can define a polynomial link invariant $$V^k_L(q) = \sum_{i,j}(-1)^i \,\text{rk}\,E_{i,j}^k(L)\, \cdot q^{j/2}$$ for each $k\geq 2$ (here, $i$ and $j$ correspond to the homological and quantum gradings, respectively). These conjectural link polynomials are generalizations of the classical Jones polynomial $V_L(q)$ in the sense that $V^2_L(q) = V_L(q)$, and that $V^k_L(q) = V_L(q)$ whenever $\kh(L)$ is supported in a single $\delta$-grading. 

In another direction, it would be interesting to determine whether link cobordisms induce well-defined maps between the higher terms in this spectral sequence, as was first suggested by Ozsv{\'a}th and Szab{\'o} in \cite{osz12}. For instance, a cobordism $Z \subset S^3 \times [0,1]$ from $L_1$ to $L_2$ induces a map from $\kh(L_1)$ to $\kh(L_2)$ \cite{kh1,jacobsson}. Similarly, the double cover of $S^3 \times [0,1]$ branched along $Z$ is a 4-dimensional cobordism from $\Sigma(L_1)$ to $\Sigma(L_2)$, and, therefore, induces a map from $\hf(-\Sigma(L_1))$ to $\hf(-\Sigma(L_2))$ \cite{osz5}. It seems very likely, in light of our invariance result, that both of these maps correspond to members of a larger family of maps $$\{E^k(Z): E^k(L_1) \rightarrow E^k(L_2)\}_{k=2}^{\infty}$$ induced by $Z$. We plan to return to this in a future paper.

In \cite{pla1}, Plamenevskaya defines an invariant of transverse links in the tight contact 3-sphere $(S^3,\xi_{rot})$ using Khovanov homology. To be precise, for a transverse link $L$, she identifies a distinguished element $\khc(L) \in \kh(L)$ which is an invariant of $L$ up to transverse isotopy. In Section \ref{sec:transv}, we show that $\khc(L)$ gives rise to a transverse invariant $\khc^k(L) \in E^k(L)$ for each $k\geq 2$ (where $\khc^2(L)$ corresponds to $\khc(L)$ under the identification of $E^2(L)$ with $\kh(L)$). It remains to be seen whether Plamenevskaya's invariant can distinguish two transversely non-isotopic knots which are smoothly isotopic and have the same self-linking number. Perhaps the invariants $\khc^k(L)$ will be more successful in this regard, though there is currently no evidence to support this hope. 

There are, however, other uses for these invariants. If $L$ is a transverse link in $(S^3,\xi_{rot})$, we denote by $\xi_L$ the contact structure on $\Sigma(L)$ obtained by lifting $\xi_{rot}$. The following proposition exploits the relationship between $\psi(L)$ and $c(\xi_L)$ discovered by Roberts in \cite{lrob1} (see \cite[Proposition 1.4]{baldpla} for comparison).

\begin{proposition}
\label{prop:vanishing} If $L$ is a transverse link for which $\khc^k(L)=0$, and $E^k(L)$ is supported in non-positive homological gradings, then the contact invariant $c(\xi_L) = 0$, and, hence, the contact structure $\xi_L$ is not strongly symplectically fillable.
\end{proposition}

The fact that $\psi(L)$ gives rise to a cycle in each term $(E^k(L),D^k)$ is also helpful in computing the invariants $E^k(L).$ We will use this principle in Section \ref{sec:example} to compute all of the terms in the spectral sequences associated to the torus knots $T(3,4)$ and $T(3,5)$. 

Since this paper first appeared, Bloom has contructed a spectral sequence from the reduced Khovanov homology of a link $L$ to (a version of) the monopole Floer homology of $-\Sigma(L)$ \cite{bloom}. Our proof of Reidemeister invariance goes through without modification to show that the higher terms in this spectral sequence are link invariants as well. It is natural to guess that this spectral sequence agrees with the one in Heegaard Floer homology.

\subsection*{Organization} Section \ref{sec:mdiags} provides a fresh review of multi-diagrams and pseudo-holomorphic polygons. In Section \ref{sec:lsss}, we outline the link surgeries spectral sequence construction in Heegaard Floer homology. Section \ref{sec:ss} describes a convenient way to think about and compute spectral sequences. In Section \ref{sec:ind}, we show that the link surgeries spectral sequence is independent of the analytic choices which go into its construction. In Section \ref{sec:review}, we describe the spectral sequence from $\kh(L)$ to $\hf(-\Sigma(L))$. Section \ref{sec:reid} constitutes the meat of this article: there, we show that this spectral sequence is invariant under the Reidemeister moves. In Section \ref{sec:transv}, we use this spectral sequence to define a sequence of transverse link invariants. Finally, in Section \ref{sec:example}, we compute the spectral sequences associated to $T(3,4)$ and $T(3,5)$.

\subsection*{Acknowledgements} I wish to thank Jon Bloom, Josh Greene, Eli Grigsby, Peter Ozsv{\'a}th, Liam Watson and Stefan Wehrli for interesting discussions, and Lawrence Roberts for helpful correspondence. I owe special thanks to Josh for help in computing the spectral sequences for the examples in Section \ref{sec:example}, and to Emmanuel Wagner who pointed out an error in the original proof of Reidemeister III invariance.

\section{Multi-diagrams and pseudo-holomorphic polygons}
\label{sec:mdiags}

In this section, we review the pseudo-holomorphic polygon construction in Heegaard Floer homology. See \cite{osz8} for more details.

A \emph{pointed multi-diagram} consists of a Heegaard surface $\Sigma$; several sets of attaching curves $\gamma^1,\dots,\gamma^r$; and a basepoint $z \in \Sigma - \cup_i \gamma^i$. Here, each $\gamma^i$ is a $g$-tuple $\{\gamma^i_1,\dots,\gamma^i_g\}$, where $g = g(\Sigma)$. For each $i$, let $\T_{\gamma^i}$ be the torus $\gamma^i_1\times\dots\times\gamma^i_g$ in the symmetric product $Sym^g(\Sigma)$. We denote by $Y_{\gamma^{i},\gamma^{j}}$ the 3-manifold with pointed Heegaard diagram $(\Sigma,\gamma^{i},\gamma^{j},z).$ 

Recall that, for $\bx$ and $\by$ in $\T_{\gamma^{i}}\cap\T_{\gamma^{j}}$, a \emph{Whitney disk} from $\bx$ to $\by$ is a map $\phi$ from the infinite strip $\mathbf{P}_2=[0,1]\times i\mathbb{R} \subset \mathbb{C}$ to $Sym^g(\Sigma)$ for which $$\phi(\{1\}\times i\mathbb{R})\subset \T_{\gamma^{j}},\,\,\,\,\,\, \,\phi(\{0\}\times i\mathbb{R})\subset \T_{\gamma^{i}}, \,\,\,\,\,\,\,\lim_{t\rightarrow -\infty}\phi(s+it) = \bx,\,\,\,\,\,\,\,\lim_{t\rightarrow +\infty}\phi(s+it) = \by. $$ The space of homotopy classes of Whitney disks from $\bx$ to $\by$ is denoted by $\pi_2(\bx,\by)$. We define \emph{Whitney $n$-gons} in a similar manner. 

For each $n\geq 3$, we fix a region $\Pn \subset \mathbb{C}$ which is conformally equivalent to an $n$-gon. We shall think of $\Pn$ as merely a topological region in $\mathbb{C}$, without a specified complex structure. We require that $\Pn$ does not have vertices, but half-infinite strips instead (so that our Whitney $n$-gons are straightforward generalizations of Whitney disks). After fixing a clockwise labeling $e_1,\dots, e_n$ of the sides of $\Pn$, we require that for each pair of non-adjacent sides $e_i$ and $e_j$ with $i<j$ there is an oriented line segment $C_{i,j}$ from $e_i$ to $e_j$ which is perpendicular to these sides. Let $s_i$ denote the half-infinite strip sandwiched between sides $e_i$ and $e_{i+1}.$ See Figure \ref{fig:sept}.(a) for an example when $n=7$.

\begin{figure}[!htbp]
\labellist 
\hair 2pt 
\small
\pinlabel $\phi$ at 837 239

\pinlabel $e_1$ at 260 80 
\pinlabel \rotatebox{-40}{$e_2$} at 135 140
\pinlabel \rotatebox{-22}{$s_1$} at 170 60 
\pinlabel \rotatebox{-60}{$s_2$} at 60 198
\pinlabel ${\gamma^{j_1}}$ at 840 70 

\pinlabel \rotatebox{-42}{${\gamma^{j_2}}$} at 708 133
\pinlabel \rotatebox{57}{${\gamma^{j_n}}$} at 972 138
\pinlabel \rotatebox{-22}{$\bx_1$} at 727 15 
\pinlabel \rotatebox{22}{$\bx_n$} at 948 18 
\pinlabel \rotatebox{-64}{$\bx_2$} at 596 186 

\pinlabel $(a)$ at 25 450 
\pinlabel $(b)$ at 600 450
\pinlabel $C_{4,7}$ at 375 419
\endlabellist 
\begin{center}
\includegraphics[height=4.8cm]{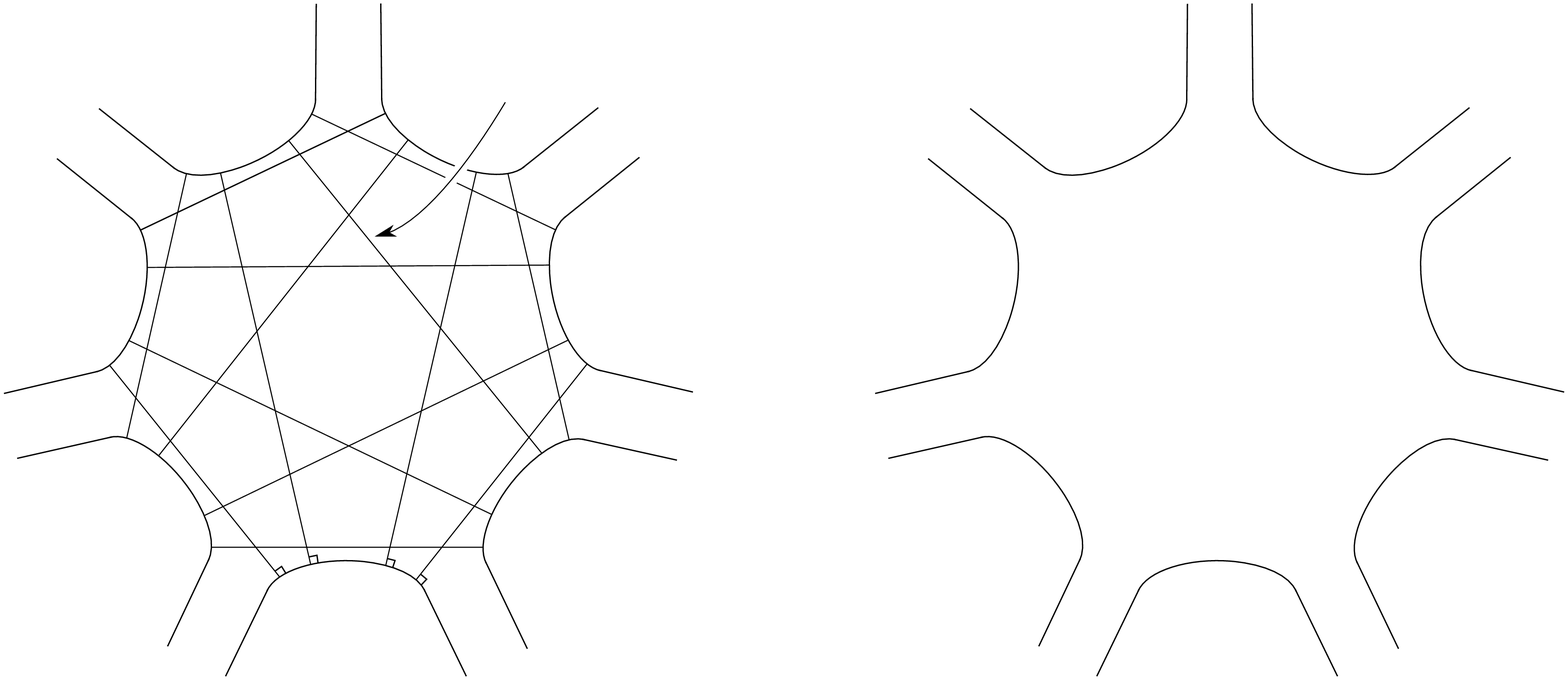}
\caption{\quad On the left, $\mathbf{P}_7$. On the right, a schematic illustration of some $\phi \in \pi_2(\bx_1,\dots,\bx_n).$ }
\label{fig:sept}
\end{center}
\end{figure}

Suppose that $\bx_1,\dots,\bx_{n-1}$ and $\bx_n$ are points in $\T_{\gamma^{j_1}} \cap \T_{\gamma^{j_2}},\,\dots, \,\T_{\gamma^{j_{n-1}}} \cap \T_{\gamma^{j_{n}}},$ and $\T_{\gamma^{j_1}} \cap \T_{\gamma^{j_n}}$. A Whitney $n$-gon connecting these points is a map $\phi:\Pn\rightarrow Sym^g(\Sigma)$ for which $\phi(e_i) \subset \T_{\gamma^{j_i}}$ for $i=1,\dots,n$, and $\phi$ sends points in $s_i$ asymptotically to $\bx_i$ as the parameter in the strip approaches $\infty$. See Figure \ref{fig:sept}.(b) for a schematic illustration. We denote the space of homotopy classes of such Whitney $n$-gons by $\pi_2(\bx_1,\dots,\bx_n)$. 


Now, fix a complex structure $\mathfrak{j}$ over $\Sigma$ and choose a contractible open set $\mathcal{U}$ of \emph{$\mathfrak{j}$-nearly-symmetric} almost-complex structures over $Sym^g(\Sigma)$ which contains $Sym^g(\mathfrak{j})$ (see \cite[Section 3]{osz8}). In addition, fix a path $J_s \subset \mathcal{U}$ for $s\in [0,1]$. We define a map $J_2:\mathbf{P}_2\rightarrow \mathcal{U}$ by $$J_2(s+it) = J_s.$$ A pseudo-holomorphic representative of $\phi\in \pi_2(\bx,\by)$ is a map $u:\mathbf{P}_2\rightarrow Sym^g(\Sigma)$ which is homotopic to $\phi$ and which satisfies the Cauchy-Riemann equation $$du_p\circ i = J_2(p) \circ du_p$$ for all $p\in \mathbf{P}_2$. Let $\wM_{J_2}(\phi)$ denote the moduli space of pseudo-holomorphic representatives of $\phi$.  This moduli space has an expected dimension, which we denote by $\mu(\phi)+1$. If $\mu(\phi) = 0$, then $\wM_{J_2}(\phi)$ is a smooth, compact 1-manifold for sufficiently generic choices of $\mathfrak{j}$ and $J_2$. Note that the $\mathbb{R}$-action on $\mathbf{P}_2$ given by translation in the $i\mathbb{R}$ direction induces an $\mathbb{R}$-action on $\wM_{J_2}(\phi)$. If $\mu(\phi)=0$, then the quotient $\Mh_{J_2}(\phi) = \wM_{J_2}(\phi)/\mathbb{R}$ is a compact 0-manifold. Recall that $(\cf(Y_{\gamma^{i}, \gamma^{j}}),\partial_{J_2})$ is the chain complex generated over $\zzt$ by intersection points $\bx \in \T_{\gamma^{i}}\cap\T_{\gamma^{j}}$ whose differential is defined by $$\partial_{J_2} (\bx) =f^{J_2}_{\gamma^{i},\gamma^{j}}(\bx)= \sum_{\by\in\T_{\gamma^{i}}\cap\T_{\gamma^{j}}}\,\,\sum_{\{ \phi\in \pi_2(\bx,\by)\,|\, \mu(\phi)=0,\, n_z(\phi)=0\}} \# (\Mh_{J_2}(\phi)) \cdot \by,$$ where $n_z(\phi)$ is the intersection of $\phi$ with the subspace $\{z\}\times Sym^{g-1}(\phi)\subset Sym^g(\Sigma)$ (we choose $\mathcal{U}$ so that $n_z(\phi)$ is always non-negative). This sum is finite as long as $(\Sigma,\gamma^{i},\gamma^{j},z)$ is \emph{weakly-admissible} (see \cite[Section 4]{osz8}). 

Further setup is required to define the maps induced by counting more general pseudo-holomorphic $n$-gons, as we would like for these maps to satisfy certain $\mathcal{A}_{\infty}$ relations. For more details on this sort of construction, see \cite{desilva,fukaya,seidel2,tian}. 

Let $\Co(\Pn)$ denote the space of conformal structures on $\Pn$. $\Co(\Pn)$ can be thought of as the space of conformal $n$-gons in $\mathbb{C}$ which are obtained from the region $\Pn$ as follows: choose $n-3$ mutually disjoint chords $C_1,\dots,C_{n-3}$ among the $C_{i,j}$, and replace each $C_k$ by the product $C_k\times I_k$ for some interval $I_k=[-t_k,t_k]$. We embed this product so that each $C_k\times \{t\}$ is parallel to and points in the same direction as the original $C_k$, and so that the vectors $i\frac{\partial}{\partial t}$ agree with the orientation of $C_k$. We refer to this process as \emph{stretching along the chords $C_1,\dots,C_{n-3}$}. An $n$-gon obtained in this way, with its conformal structure inherited from $\mathbb{C}$, is called a \emph{standard conformal $n$-gon}. See Figure \ref{fig:con} for an illustration. 

$\Co(\Pn)$ has a compactification $\hCo(\Pn)$ gotten by including the degenerate $n$-gons obtained in the limit as one or more of the stretching parameters $t_k$ approaches $\infty$. For each set of $n-3$ disjoint chords, this stretching procedure (allowing for degenerations) produces an $(n-3)$-dimensional cube of conformal $n$-gons since each of the $n-3$ stretching parameters can take any value in $[0,\infty]$. The hypercubes formed in this way fit together into a convex polytope $\mathscr{P}_n$ called an \emph{associahedron} (see \cite{bloom} for a related construction). The interior of $\mathscr{P}_n$ parametrizes $\Co(\Pn)$, and its boundary parametrizes the degenerate $n$-gons. More precisely, points in the interiors of the $k$-dimensional facets of $\mathscr{P}_n$ correspond to degenerate $n$-gons in which exactly $k-n+3$ of the stretching parameters are infinite.

\begin{figure}[!htbp]
\labellist 
\hair 2pt 
\small
\pinlabel {$I_1$} at 263 143
\pinlabel {$e_1$} at 400 63
\pinlabel \rotatebox{-12}{$I_2$} at 156 227
\pinlabel \rotatebox{38}{$I_3$} at 356 513
\pinlabel \rotatebox{13}{$I_4$} at 574 524

\endlabellist 
\begin{center}
\includegraphics[height=5.8cm]{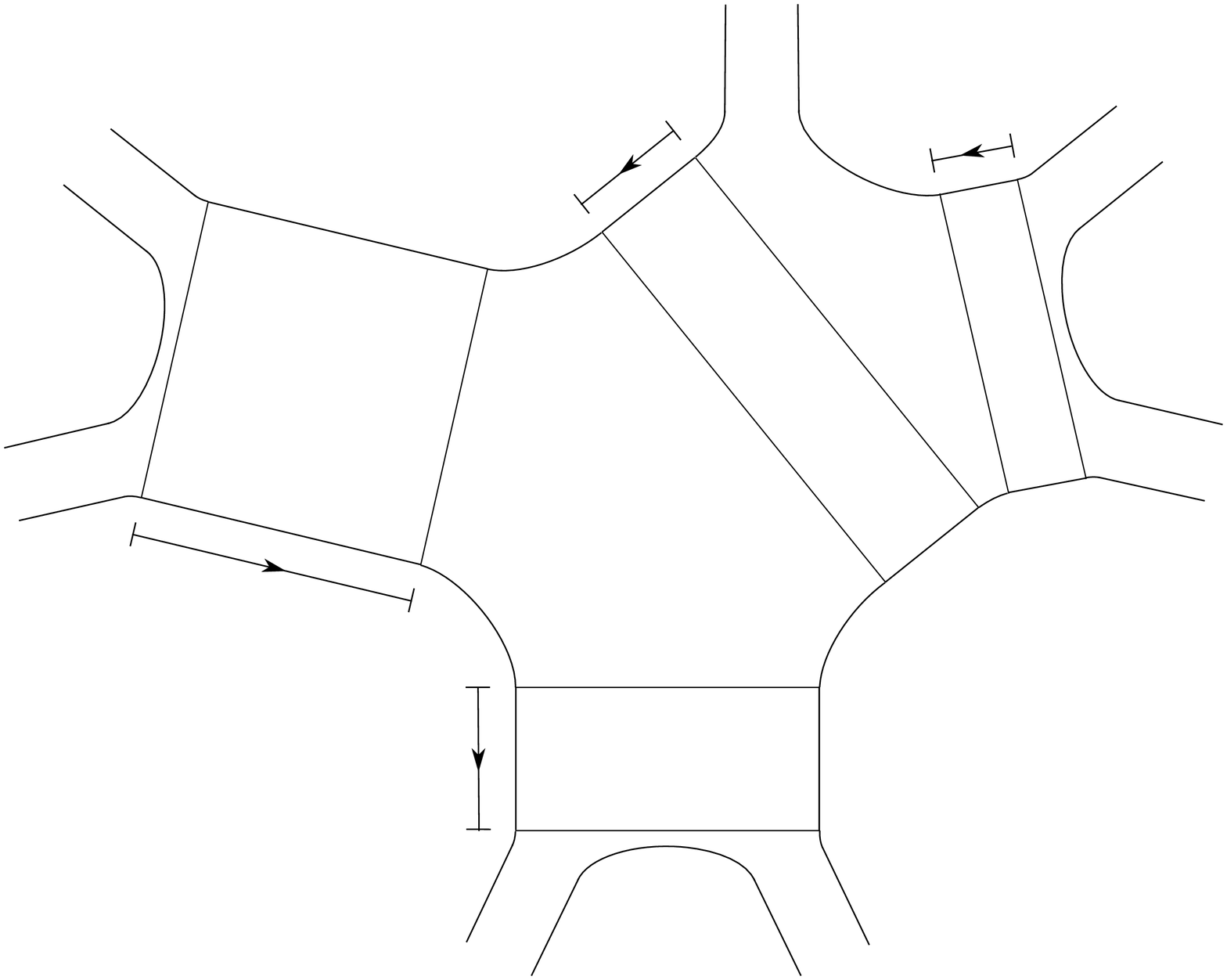}
\caption{\quad A standard conformal $n$-gon obtained by stretching $\mathbf{P}_7$ along the disjoint chords $C_1=C_{2,7}$, $C_2=C_{2,4}$, $C_3=C_{4,7}$ and $C_4=C_{5,7}$. The arrows on the intervals indicate the directions of increasing $t$ values.}
\label{fig:con}
\end{center}
\end{figure}

Following Ozsv{\'a}th and Szab{\'o} \cite[Section 8]{osz8}, we construct maps $J_n:\mathbf{P}_n\times\Co(\Pn)\rightarrow\mathcal{U}$ which satisfy certain compatibility conditions. To describe these conditions, it is easiest to think of $J_n$ as a collection of maps on the standard conformal $n$-gons. If $P$ is a standard conformal $n$-gon, we first require that $J_n|_P$ agrees with $J_2$ on the strips $s_1,\dots,s_n$, via conformal identifications of these strips with $[0,1]\times i[0,\infty)\subset \mathbf{P}_2$. 

Next, suppose that $P$ is obtained from $\Pn$ by stretching along some set of disjoint chords $C_1,\dots, C_{n-4}$. Suppose that $C_{i,j}$ is disjoint from these chords and consider the $n$-gon $P_T$ obtained from $P$ by replacing $C_{i,j}$ by the product $C_{i,j}\times[-T,T]$. Then $C_{i,j}\times\{0\}$ divides $P_T$ into two regions. Of these two regions, let $L_T$ denote the one which contains $C_{i,j}\times\{-T\}$, and let $R_T$ denote the one which contains $C_{i,j}\times\{T\}$. Let $L'_T$ denote the region in $\mathbb{C}$ formed by taking the union of $L_T$ with the strip $C_{i,j}\times[0,\infty)$ along $C_{i,j}\times\{0\}$; likewise, let $R'_T$ denote the region formed by taking the union of $R_T$ with the strip $C_{i,j}\times(-\infty,0]$ along $C_{i,j}\times\{0\}$. $L'_T$ is conformally equivalent to a $(j-i+1)$-gon while $R_T'$ is conformally equivalent to a $(n-j+i+1)$-gon. 
Let $s_L$ denote the half-infinite strip $C_{i,j}\times [-T,\infty) \subset L_T'$, and let $s_R$ denote the strip $C_{i,j}\times (-\infty,T]\subset R_T'$. See Figure \ref{fig:deg} for an illustration. 

\begin{figure}[!htbp]
\labellist 
\hair 2pt 
\small
\pinlabel {$(a)$} at 70 645
\pinlabel {$(b)$} at 870 645
\pinlabel {$s_R$} at 1330 440
\pinlabel {$s_L$} at 1600 250
\pinlabel \rotatebox{38}{$-T$} at 395 592
\pinlabel \rotatebox{38}{$0$} at 298 514
\pinlabel \rotatebox{38}{$T$} at 193 433
\pinlabel \rotatebox{38}{$0$} at 1088 511
\pinlabel \rotatebox{38}{$T$} at 983 432
\pinlabel \rotatebox{38}{$-\infty$} at 1324 686
\pinlabel \rotatebox{38}{$-T$} at 1945 262
\pinlabel \rotatebox{38}{$0$} at 1850 190
\pinlabel \rotatebox{38}{$\infty$} at 1605 5
\endlabellist 
\begin{center}
\includegraphics[height =5.2cm]{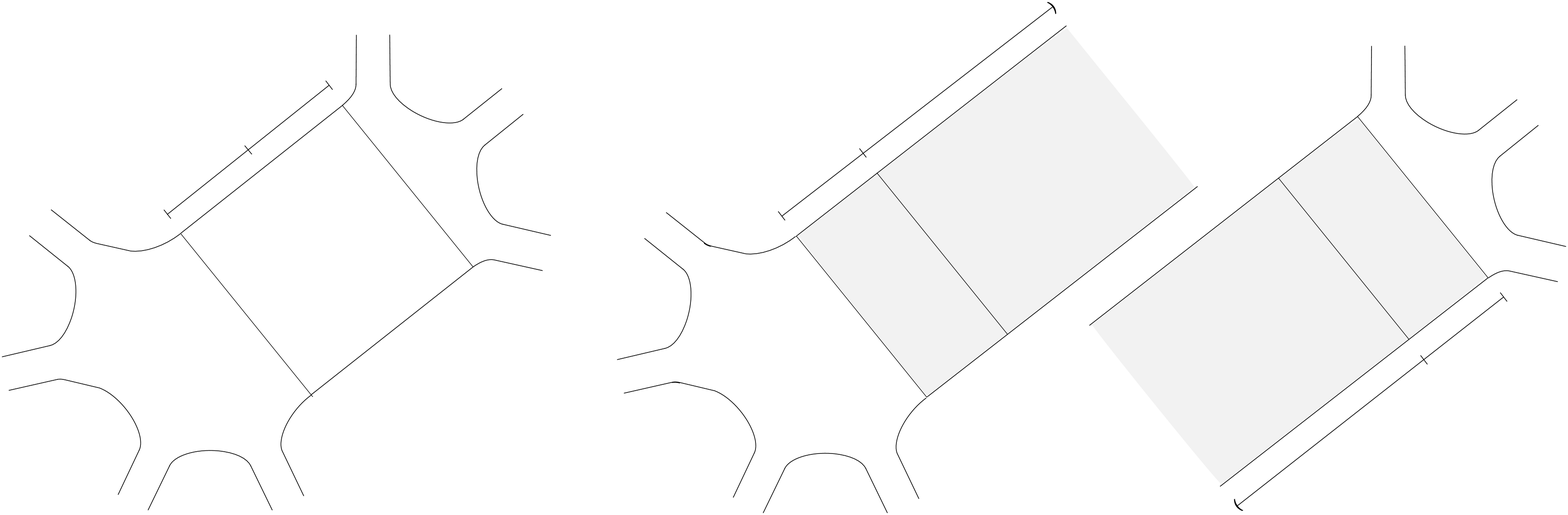}
\caption{\quad Here, $n=7$ and we have replaced $C_{4,7}$ by $C_{4,7}\times[-T,T]$. On the left, $P_T$. On the right, the conformal $5$- and $4$-gons, $R_T'$ and $L_T'$, respectively; $s_R$ and $s_L$ are the shaded strips.}
\label{fig:deg}
\end{center}
\end{figure}

Let $\Phi_L$ be the conformal equivalence which maps $L'_T$ to a standard conformal $(j-i+1)$-gon such that $s_i$ is sent to $s_1$; and let $\Phi_R$ be the conformal equivalence which maps $R'_T$ to a standard conformal $(n-j+i+1)$-gon such that $s_R$ is sent to $s_i$. The compatibility we require is that, for $T$ large enough, the restriction of $J_n$ to ${L_T}\subset P_T$ agrees with the restriction of $J_{j-i+1}\circ\Phi_L$ to $L_T\subset L'_T$; and that the restriction of $J_n$ to ${R_T}\subset P_T$ agrees with the restriction of $J_{n-j+i+1}\circ\Phi_R$ to $R_T\subset R'_T.$ Said more informally: in the limit as $T$ approaches $\infty$, the $n$-gon $P_T$ breaks apart into a $(j-i+1)$-gon and an $(n-j+i+1)$-gon, and we require that the map $J_n$ agrees in this limit with the maps $J_{j-i+1}$ and $J_{n-j+i+1}$ on the two pieces. 

Here, we are specifying the behavior of $J_n$ near the boundary of $\hCo(\Pn)$. Since $\mathcal{U}$ is contractible, it is possible to extend any map defined near this boundary to the interior. In this way, the maps $J_n$ are defined inductively beginning with $n=3$, in which case there are no degenerations and the only requirement is that $J_3$ agrees with $J_2$ on the strips $s_1,s_2,s_3$, as described above. Fix some generic $J_3$ which satisfies this requirement, and proceed. 

A pseudo-holomorphic representative of a Whitney $n$-gon $\phi\in\pi_2(\bx_1,\dots,\bx_n)$ consists of a conformal structure $\mathfrak{c}\in \Co(\Pn)$ together with a map $u:\mathbf{P}_n\rightarrow Sym^g(\Sigma)$ which is homotopic to $\phi$ and which satisfies $$du_p\circ i_{\mathfrak{c}} = J_n((p,\mathfrak{c}))\circ du_p$$ for all $p\in\mathbf{P}_n,$ where $i_{\mathfrak{c}}$ is the almost-complex structure on $\Pn$ associated to $\mathfrak{c}$. Let $\mathcal{M}_{J_n}(\phi)$ denote the moduli space of pseudo-holomorphic representatives $(u,\mathfrak{c})$ of $\phi$. As before, this moduli space has an expected dimension, which we denote by $\mu(\phi)$. If $\mu(\phi) = 0$ then $\mathcal{M}_{J_n}(\phi)$ is a compact 0-manifold for sufficiently generic choices of $\mathfrak{j}$ and $\{J_n\}_{n\geq2}$. 
 With this fact in hand, we may define a map $$f^{J_n}_{\gamma^{j_1},\dots,\gamma^{j_n}}:\cf(Y_{\gamma^{j_1},\gamma^{j_2}})\otimes\dots\otimes\cf(Y_{\gamma^{j_{n-1}},\gamma^{j_n}})\rightarrow \cf(Y_{\gamma^{j_1},\gamma^{j_n}})$$ by $$f^{J_n}_{\gamma^{j_1},\dots,\gamma^{j_n}}(\bx_1\otimes\dots\otimes\bx_{n-1}) = \sum_{\bx_n\in\T_{\gamma^{j_1}}\cap\T_{\gamma^{j_n}}}\,\,\sum_{\{ \phi\in \pi_2(\bx_1,\dots,\bx_n)\,|\, \mu(\phi)=0,\, n_z(\phi)=0\}} \# (\mathcal{M}_{J_n}(\phi)) \cdot \bx_n.$$ This sum is finite as long as the associated pointed multi-diagram is weakly-admissible. Since weak-admissibility can easily be achieved via ``winding'' (as in \cite[Section 5]{osz8}), we will not worry further about such admissibility concerns, in this or any of the following sections. Also, we will generally omit the superscript $J_n$ from the notation for these maps unless we wish to emphasize our choice of analytic data. 
 


\section{The link surgeries spectral sequence}
\label{sec:lsss}

This section provides a review of the link surgeries spectral sequence, as defined by Ozsv{\'a}th and Szab{\'o} in \cite{osz12}.

Suppose that $\bL = L_1\cup\dots\cup L_m$ is a framed link in a 3-manifold $Y$. A \emph{bouquet} for $\bL$ is an embedded 1-complex $\Gamma \subset Y$ formed by taking the union of $\bL$ with $m$ paths joining the components $L_i$ to a fixed reference point in $Y$. The regular neighborhood $N(\Gamma)$ of this bouquet is a genus $m$ handlebody, and a subset of its boundary is identified with the disjoint union of $m$ punctured tori $F_i$, one for each component of $\bL$.

\begin{definition}
A pointed Heegaard diagram $(\Sigma,\alpha,\beta,z)$ is \emph{subordinate} to $\Gamma$ if 
\begin{enumerate}
\item $(\Sigma, \{\alpha_1,\dots,\alpha_g\}, \{\beta_{m+1},\dots,\beta_g\})$ describes the complement $Y-\text{int}N(\Gamma),$ 
\item after surgering $\beta_{m+1},\dots,\beta_g$ from $\Sigma$, $\beta_i$ lies in $F_i\subset \partial N(\Gamma)$ for each $i = 1,\dots,m,$
\item for $i=1,\dots,m$, the curve $\beta_i$ is a meridian of $L_i.$
\end{enumerate}
\end{definition}

Suppose that $(\Sigma,\alpha,\beta,z)$ is \emph{subordinate} to $\Gamma$, and fix an orientation on the meridians $\beta_1,\dots,\beta_m$. For $i=1,\dots,m$, let $\gamma_i$ be an oriented curve in $F_i\subset \Sigma$ which represents the framing of the component $L_i$, such that $\beta_i$ intersects $\gamma_i$ in one point with algebraic intersection number $\beta_i\cdot \gamma_i=1$ and $\gamma_i$ is disjoint from all other $\beta_j$. For $i=1,\dots,m$, let $\delta_i$ be an oriented curve in $F_i\subset \Sigma$ which is homologous to $\beta_i+\gamma_i$, such that $\delta_i$ intersects each of $\beta_i$ and $\gamma_i$ in one point and is disjoint from all other $\beta_j$ and $\gamma_j$. For $i=m+1,\dots,g$, we define $\gamma_i$ and $\delta_i$ to be small Hamiltonian translates of $\beta_i$.

For $I = (I_1,\dots,I_m)\in \{\infty,0,1\}^m$, let $\eta(I) = \{ \eta(I)_1,\dots,\eta(I)_g\}$ be the set of attaching curves on $\Sigma$ defined by 
$$\eta(I)_j = 
\left\{\begin{array}{ll}
\beta_j &\text{if $I_j=\infty,$}\\
\gamma_j &\text{if $I_j=0,$}\\
\delta_j &\text{if $I_j=1.$}
\end{array}\right.$$
Then $Y_{\alpha,\eta(I)}$ is the 3-manifold obtained from $Y$ by performing $I_j$-surgery on the component $L_j$ for $j=1,\dots,m$. To define a multi-diagram using the sets of attaching curves $\eta(I)$, we need to be a bit more careful. Specifically, in the definition of the $g$-tuples $\eta(I)$, we must use different (small) Hamiltonian translates of the $\beta$, $\gamma$ and $\delta$ curves for each $I$, for which none of the associated isotopies cross the basepoint $z$. In addition, we require that every translate of a $\beta$ (resp. $\gamma$ resp. $\delta$) curve intersects each of the other copies in two points, and that we can join the basepoint $z$ to a point on each of $\gamma_1,\dots,\gamma_m$ by a path which does not intersect the $\alpha$ curves or the other $\gamma$ curves. We say that the pointed multi-diagram consisting of the Heegaard surface $\Sigma$; the sets of attaching curves $\alpha$ and the $\eta(I)$ for $I\in\{\infty,0,1\}^m$; and the basepoint $z$ is \emph{compatible} with the framed link $\bL$. It is not hard to see that such a compatible multi-diagram exists.

Order the set $\{\infty,0,1\}$ by $\infty<0<1$. We say that $I\leq I'$ if $I_j \leq I'_j$ for every $j \in \{1,\dots,m\}$; if $I\leq I'$ and $I\neq I'$, then we write $I<I'$. The vector $I'$ is called an \emph{immediate successor} of $I$ if $I_j = I_j'$ for all $j \neq k$ and either $I_j = \infty$ and $I_j'=0$ or $I_j=0$ and $I_j'=1$. For $I \leq I'$, a \emph{path} from $I$ to $I'$ is a sequence of inequalities $I = I^2<\dots<I^n = I'$, where each $I^j$ is the immediate successor of $I^{j-1}$.\footnote{We start with $I=I^2$ rather than, say, $I=I^1$ for the sake of notational convenience down the road.} In particular, we think of $I=I^2=I$ as the ``trivial" path from $I$ to $I$. For the rest of this section, we shall only consider vectors in $\{0,1\}^m$.

Let $X$ be the vector space defined by $$X= \bigoplus_{I\in \{0,1\}^m}\cf(Y_{\alpha,\eta(I)}).$$ For every path $ I^2<\dots<I^n$, we define a map $$D_{I^2<\dots<I^n}: \cf(Y_{\alpha,\eta(I^2)})\rightarrow\cf(Y_{\alpha, \eta(I^n)})$$ by $$D_{I^2<\dots<I^n}(\bx) = f_{\alpha, \eta(I^2),\dots,\eta(I^n)}(\bx\otimes\htheta_{2,3}\otimes\dots\otimes\htheta_{n-1,n}),$$ where $\htheta_{i,j}$ is the unique intersection point in $ \T_{\eta(I^i)} \cap \T_{\eta(I^j)}$ with highest Maslov grading when thought of as a generator in $\cf(Y_{\eta(I^i),\eta(I^{j})})$. We define $D_{I,I'}$ to be the sum over paths from $I$ to $I'$, $$D_{I,I'}=\sum_{I = I^2<\dots<I^n = I'} D_{I^2<\dots<I^n}.$$ Note that $D_{I,I} = D_{I^2}$ is just the standard Heegaard Floer boundary map for the chain complex $(\cf(Y_{\alpha,\eta(I)}),\partial).$ Finally, we define $D:X\rightarrow X$ to be the sum $$D = \sum D_{I,I'}.$$

\begin{proposition}[\cite{osz12}]
\label{prop:bound}
$D$ is a boundary map; that is, $D^2 = 0$. 
\end{proposition}

In particular, $(X,D)$ is a chain complex. The proof of Proposition \ref{prop:bound} proceeds roughly as follows. For intersection points $\bx \in \T_{\alpha}\cap\T_{\eta(I)}$ and $\by \in \T_{\alpha}\cap\T_{\eta(I')}$, the coefficient of $\by$ in $D^2(\bx)$ counts maps of certain degenerate $n$-gons into $Sym^g(\Sigma)$. These maps can be thought of as ends of a union of 1-dimensional moduli spaces (c.f. \cite{desilva,fukaya,seidel2,tian}), \begin{equation}\label{eqn:union}\coprod_{p\in S}\,\,\coprod_{\phi\in \pi_2(\bx,\by;p)}\mathcal{M}_{J_{n}}(\phi),\end{equation} where $S$ is the set of paths from $I$ to $I'$, and $\pi_2(\bx,\by;p)$ denotes, for a path $p$ in $S$ given by $I=I^2<\dots<I^n=I'$, the set of $\phi \in\pi_2(\bx,\htheta_{2,3},\dots,\htheta_{n-1,n},\by)$ with $n_z(\phi)=0$ and $\mu(\phi) = 1.$ This follows, in part, from the compatibility conditions described in the previous section, which control the behaviors of pseudo-holomorphic representatives of these $\phi$ as their underlying conformal structures degenerate. 

In general, for a fixed path $p\in S$, the union $$\coprod_{\phi\in \pi_2(\bx,\by;p)}\mathcal{M}_{J_{n}}(\phi)$$ has ends which do not contribute to the coefficient of $\by$ in $D^2(\bx)$; but these cancel in pairs when we take a union over all paths in $S$. Proposition \ref{prop:bound} then follows from the fact that the total number of ends of the union in (\ref{eqn:union}) is zero mod 2. We will revisit this sort of argument in more detail in Section \ref{sec:ind}.

The beautiful result below is the basis for our interest in these constructions.

\begin{theorem}[{\rm \cite[Theorem 4.1]{osz12}}]
\label{thm:conv}
The homology $H_*(X,D)$ is isomorphic to $\hf(Y)$.
\end{theorem}

There is a grading on $X$ defined, for $\bx \in \cf(Y_{\alpha,\eta(I)})$, by $h'(\bx) = I_1+\dots+I_m.$ This grading gives rise to a filtration of the complex $(X,D)$. This filtration, in turn, gives rise to what is known as \emph{the link surgeries spectral sequence}. The portion $D^0$ of $D$ which does not increase $h'$ is just the sum of the boundary maps $D_{I,I}$. Therefore, $$E^1(X) \cong \bigoplus_{I\in \{0,1\}^m} \hf(Y_{\alpha,\eta(I)}).$$ Observe that the spectral sequence differential $D^1$ on $E^1(X)$ is the sum of the maps $$(D_{I,I'})_*:\hf(Y_{\alpha,\eta(I)})\rightarrow \hf(Y_{\alpha,\eta(I')})$$ over all pairs $I,I'$ for which $I'$ is an immediate successor of $I$. If $I'$ is an immediate successor of $I$ with $I_j<I_j'$, then $Y_{\alpha,\eta(I')}$ is obtained from $Y_{\alpha,\eta(I)}$ by changing the surgery coefficient on $L_j$ from $0$ to $1$, and $(D_{I,I'})_*$ is the map induced by the corresponding 2-handle cobordism.  By Theorem \ref{thm:conv}, this spectral sequence converges to $E^{\infty} = \hf(Y)$.

This spectral sequence behaves nicely under connected sums. Specifically, if $\bL_1\in Y_1$ and $\bL_2\in Y_2$ are framed links, then a pointed multi-diagram for the disjoint union $$\bL_1\cup \bL_2\in Y_1 \# Y_2$$ can be obtained by taking a connected sum of pointed multi-diagrams for $\bL_1$ and $\bL_2$ near their basepoints, and then replacing these two basepoints with a single basepoint contained in the connected sum region. Let $(X_1,D_1)$ and $(X_2,D_2)$ denote the complexes associated to the multi-diagrams for $\bL_1$ and $\bL_2$, and let $(X_{\#}, D_{\#})$ denote the complex associated to the multi-diagram for $\bL_1 \cup \bL_2$, obtained as above. The lemma below follows from standard arguments in Heegaard Floer homology (see \cite[Section 6]{osz14}). 

\begin{lemma}
\label{lem:connect}
$(X_{\#},D_{\#})$ is the tensor product $(X_1,D_1)\otimes(X_2,D_2)$. In particular, $E^k(X_{\#}) = E^k(X_1)\otimes E^k(X_2)$.
\end{lemma}

\section{Computing spectral sequences}
\label{sec:ss}
In this section, we provide a short review of the ``cancellation lemma," and describe how it is used to think about and compute spectral sequences.

\begin{lemma}[\rm{see \cite[Lemma 5.1]{ras}}] 
\label{lem:cancel}
Suppose that $(C,d)$ is a complex over $\zzt$, freely generated by elements $x_i,$ and let $d(x_i,x_j)$ be the coefficient of $x_j$ in $d(x_i)$. If $d(x_k,x_l)=1,$ then the complex $(C',d')$ with generators $\{x_i| i \neq k,l\}$ and differential $$d'(x_i) = d(x_i) + d(x_i,x_l)d(x_k)$$ is chain homotopy equivalent to $(C,d)$. The chain homotopy equivalence is induced by the projection $\pi:C\rightarrow C'$, while the equivalence $\iota: C'\rightarrow C$ is given by $\iota(x_i) = x_i + d(x_i,x_l)x_k$.
\end{lemma}

 We say that $(C',d')$ is obtained from $(C,d)$ by ``canceling" the component of the differential $d$ from $x_k$ to $x_l$. Lemma \ref{lem:cancel} admits a refinement for filtered complexes. In particular, suppose that there is a grading on $C$ which induces a filtration of the complex $(C,d)$, and let the elements $x_i$ be homogeneous generators of $C$. If $d(x_k,x_l) = 1$, and $x_k$ and $x_l$ have the same grading, then the complex obtained by canceling the component of $d$ from $x_k$ to $x_l$ is \emph{filtered} chain homotopy equivalent to $(C,d)$ since both $\pi$ and $\iota$ are filtered maps in this case. 
 
  \begin{figure}[!htbp]

 \labellist 
\small\hair 2pt 
\pinlabel $(a)$ at 69 190 
\pinlabel $(b)$ at 298 190
\pinlabel $(c)$ at 527 190 
\pinlabel $(d)$ at 759 190
\pinlabel $1$ [r] at 12 150
\pinlabel $0$ [r] at 12 78 
\pinlabel $-1$ [r] at 12 8
\pinlabel $x$ at 119 140
\pinlabel $y$ at 409 65
\pinlabel $z$ at 567 36
\pinlabel $w$ at 777 79

\endlabellist 
\begin{center}
\includegraphics[width=12cm]{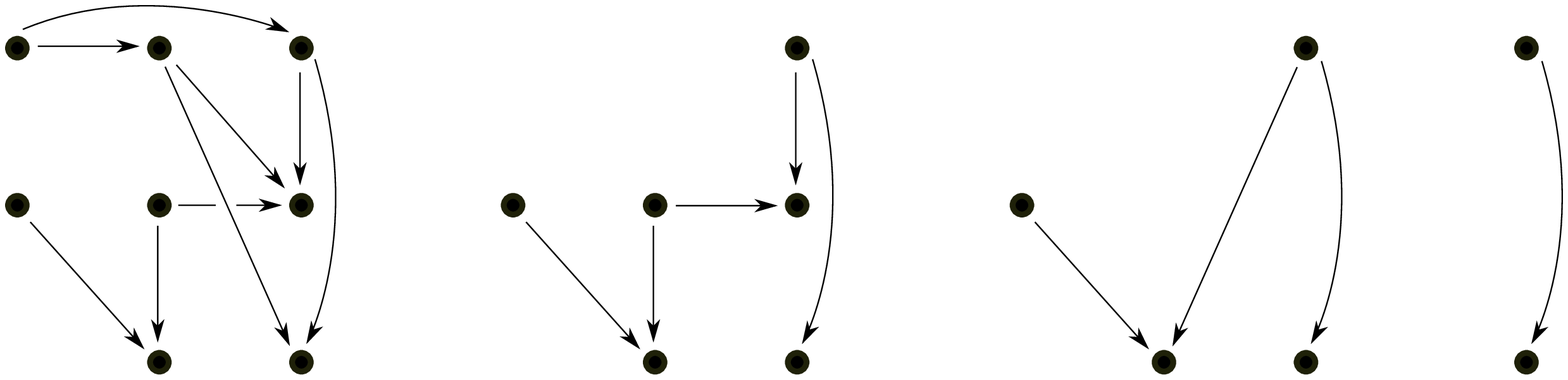}
\caption{\quad The diagram in $(a)$ represents a graded complex $C$, where the grading of a generator is given by $1$, $0$, or $-1$. This grading induces a filtration $\mathcal{F}_{-1} \subset \mathcal{F}_0 \subset \mathcal{F}_1 = C$. The complex in $(b)$ is obtained from that in $(a)$ by canceling the component $x$ of the differential. The complex in $(c)$ is obtained from that in $(b)$ by canceling $y$. This graded vector space represents the $E^1$ term of the spectral sequence associated to the filtration of $C$. The complex in $(d)$ is obtained from that in $(c)$ by canceling $z$, and it represents the $E^2$ term of the spectral sequence. The $E^3 = E^{\infty}$ term of the spectral sequence is trivial, and is obtained from the complex in $(d)$ by canceling $w$.}
\label{fig:complex}
\end{center}
\end{figure}

Computing the spectral sequence associated to such a filtration is the process of performing cancellation in a series of stages until we arrive at a complex in which the differential is zero (the $E^{\infty}$ term). The $E^n$ term records the result of this cancellation after the $n$th stage. Specifically, the $E^0$ term is simply the graded vector space $C = \bigoplus C_i$. The $E^1$ term is the graded vector space $C^{(1)}$, where $(C^{(1)},d^{(1)})$ is obtained from $(C,d)$ by canceling the components of $d$ which do not shift the grading. For $n>1$, the $E^n$ term is the graded vector space $C^{(n)}$, where $(C^{(n)},d^{(n)})$ is obtained from $(C^{(n-1)},d^{(n-1)})$ by canceling the components of $d^{(n-1)}$ which shift the grading by $n-1$. Though it is implicit here, the spectral sequence differential $d^k$ is the sum of the components of $d^{(k)}$ which shift the grading by $k$. See Figure \ref{fig:complex} for an illustration of this process (in this diagram, the generators are represented by dots and the components of the differential are represented by arrows).

Now, suppose that $F:(C_1,d_1) \rightarrow (C_2,d_2)$ is a filtered chain map, and let $E^n(C_j)$ denote the $n$th term in the spectral sequence associated to the filtration of $(C_j,d_j)$. Every time we cancel a component of $d_1$ or $d_2$, we may adjust the components of $F$ as though they were components of a differential (in fact, they \emph{are} components of the mapping cone differential). In this way, we obtain an adjusted map $F^{(n)}:(C_1^{(n)},d_1^{(n)}) \rightarrow (C_2^{(n)},d_2^{(n)})$ for each $n \geq 1$. The map from $E^n(C_1)$ to $E^n(C_2)$ induced by $F$ is, by definition, the sum of the components of $F^{(n)}$ which do not shift the grading. With this picture in mind, the following well-known fact is easy to verify.

\begin{lemma}
\label{lem:filt}
If $F:(C_1,d_1) \rightarrow (C_2,d_2)$ is a filtered chain map which induces an isomorphism from $E^n(C_1)$ to $E^n(C_2)$, then $F$ induces an isomorphism from $E^k(C_1)$ to $E^k(C_2)$ for all $k\geq n$. If $n=1$ in the statement above, then $(C_1,d_1)$ is filtered chain homotopy equivalent to $(C_2,d_2)$.
\end{lemma}

\section{Independence of multi-diagram and analytic data}
\label{sec:ind}

The complex $(X,D)$ associated to a framed link $\bL\subset Y$ depends on a choice of multi-diagram compatible with $\bL$ and a choice of auxiliary analytic data -- namely, the complex structure $\mathfrak{j}$ and the maps $\{J_n\}_{n\geq 2}$. However, we shall see in this section that the filtered chain homotopy type of $(X,D)$ does not.

\begin{theorem}[\cite{lrob3}]
\label{thm:multi}The filtered chain homotopy type of $(X,D)$ is independent of the choice multi-diagram compatible with $\bL$.
\end{theorem}

To see this, fix some choice of analytic data and suppose that $H$ and $H'$ are two multi-diagrams compatible with $\mathbb{L}$, giving rise to the complexes $(X,D)$ and $(X',D')$. Then $H$ and $H'$ are related by a sequence of isotopies, handleslides, and (de-) stabilizations. Roberts proves that there is an $h'$-filtered chain map $\Phi:(X,D) \rightarrow (X',D')$ corresponding to this sequence of operations which induces an isomorphism between the $E^1$ terms of the associated spectral sequences \cite[Section 7]{lrob3}. It follows that $(X,D)$ is filtered chain homotopy equivalent to $(X',D')$ by Lemma \ref{lem:filt}. In particular, $E^k(X) \cong E^k(X')$ for all $k\geq 1$. 

In this section, we prove the following theorem, which has been known to experts for some time, but has not been written down rigorously (though Roberts offers a sketch in \cite{lrob3}). Our proof is similar in spirit to Roberts' proof of Theorem \ref{thm:multi}.

\begin{theorem}
\label{thm:analytic}
The filtered chain homotopy type of $(X,D)$ is independent of the choice of analytic data $\mathfrak{j}$ and $\{J_n\}_{n\geq 2}$.
\end{theorem}

\begin{remark}
There is no direct analogue of a multi-diagram for the corresponding spectral sequence in monopole Floer homology (see \cite{bloom}). In that setting, the spectral sequence associated to a framed link depends \emph{a priori} only on a choice of metric and perturbation of the Seiberg-Witten equations on some 4-dimensional cobordism. And Bloom shows that, up to isomorphism, the spectral sequence is independent of these choices. 
\end{remark}

The rest of this section is devoted to proving Theorem \ref{thm:analytic}. Let $H$ be a pointed multi-diagram compatible with $\bL=L_1\cup\dots\cup L_m$, consisting of the Heegaard surface $\Sigma$; attaching curves $\alpha$ and $\eta(I)$ for $I\in \{\infty,0,1\}^m$; and a basepoint $z$. Let $H'$ be the pointed multi-diagram obtained from $H$ by replacing each curve $\eta(I)_j$ by a small Hamiltonian translate $\lambda(I)_j$. Now, fix a complex structure $\mathfrak{j}$ and suppose we are given two sets of analytic data, $\{J_n^{(0)}\}_{n\geq 2}$ and $\{J_n^{(1)}\}_{n\geq 2}$. We would like to construct a chain map from the complex $(X,D)$, associated to $H$ and the first set of analytic data, to the complex $(X',D')$, associated to $H'$ and the second set of analytic data. To do so, we first define maps $$J_{n,k}:\bP_n\times\Co(\Pn)\rightarrow \mathcal{U}$$ for $n\geq 3$ and $1\leq k\leq n$ which satisfy compatibility conditions similar to those described in Section \ref{sec:mdiags}. As in Section \ref{sec:mdiags}, we shall think of $J_{n,k}$ as a collection of maps defined on the standard conformal $n$-gons. 

If $P$ is a standard conformal $n$-gon, we first require that $J_{n,k}|_P$ agrees with $J^{(0)}_2$ on the strips $s_1,\dots,s_{k-1}$ and with $J^{(1)}_2$ on the strips $s_k,\dots,s_n$, and that $J_{n,1} = J_n^{(1)}$. Next, suppose that $P$ is obtained from $\Pn$ by stretching along some set of disjoint chords $C_1,\dots, C_{n-4}$. Suppose that $C_{i,j}$ is disjoint from these chords and consider the $n$-gon $P_T$ obtained from $P$ by replacing $C_{i,j}$ by the product $C_{i,j}\times[-T,T]$. The remaining compatibility conditions, which are to be satisfied for large $T$, are divided into three cases.

\begin{enumerate}
\item For $i<k$ and $j\leq k$, the restriction $J_{n,k}\,|\,_{{L_T}}$ agrees $J^{(0)}_{j-i+1}\circ\Phi_L\,|\,_{L_T}$; and the restriction $J_{n,k}\,|\,_{R_T}$ agrees with $J_{n-j+i+1, k-j+i+1}\circ \Phi_R\,|\,_{R_T}.$\\

\item For $i<k$ and $j>k$, the restriction $J_{n,k}\,|\,_{L_T}$ agrees with $J_{j-i+1,k-i+1}\circ\Phi_L\,|\,_{L_T}$; and the restriction $J_{n,k}\,|\,_{R_T}$ agrees with $J_{n-j+i+1, i}\circ \Phi_R\,|\,_{R_T}.$\\

\item For $i\geq k$, the restriction $J_{n,k}\,|\,_{L_T}$ agrees with $J^{(1)}_{n-j+i+1}\circ\Phi_L\,|\,_{L_T}$; and the restriction $J_{n,k}\,|\,_{R_T}$ agrees with $J_{n-j+i+1, k}\circ \Phi_R\,|\,_{R_T}.$ 
\end{enumerate}

We shall see in a bit where these conditions come. For the moment, consider the pointed multi-diagram $(\Sigma, \gamma^1,\dots,\gamma^r, z)$, and let $\bx_1,\dots,\bx_{n-1}$ and $\bx_n$ be intersection points in $\T_{\gamma^{j_1}}\cap \T_{\gamma^{j_2}}, \dots, \T_{\gamma^{j_{n-1}}}\cap \T_{\gamma^{j_n}}$ and $\T_{\gamma^{j_1}}\cap \T_{\gamma^{j_n}}$, respectively. For $\phi \in \pi_2(\bx_1,\dots,\bx_n)$, we denote by $\Mh_{J_{n,k}}(\phi)$ the moduli space of pairs $(u,\mathfrak{c})$, where $u:\Pn\rightarrow Sym^g(\Sigma)$ is homotopic to $\phi$ and $$du_p\circ i_{\mathfrak{c}} = J_{n,k}(p,\mathfrak{c})\circ du_p$$ for all $p \in \Pn$. We may then define a map $$f^{J_{n,k}}_{\gamma^{j_1},\dots,\gamma^{j_n}}:\cf(Y_{\gamma^{j_1},\gamma^{j_2}})\otimes\dots\otimes\cf(Y_{\gamma^{j_{n-1}},\gamma^{j_n}})\rightarrow \cf(Y_{\gamma^{j_1},\gamma^{j_n}})$$ exactly as in Section \ref{sec:mdiags}. 

To construct a chain map from $(X,D)$ to $(X',D')$, it is convenient to think of the g-tuples $\eta(I)$ and $\lambda(I)$ as belonging to one big multi-diagram. For $I\in\{0,1\}^{m+1}$, let $\wt I$ be the vector $(I_1,\dots,I_m) \in \{0,1\}^n$, and define 
$$\Lambda(I) = 
\left\{\begin{array}{ll}
\eta(\wt I) &\text{if $I_{m+1}=0,$}\\
\nu(\wt I) &\text{if $I_{m+1} = 1.$}
\end{array}\right.$$
For $I\leq I' \in \{0,1\}^{m+1}$ with $I_{m+1} = 0$ and $I'_{m+1} = 1$, we shall write $$I=I^2<\dots<\wh {I^k}<\dots<I^n=I'$$ to convey the fact that the last coordinate of $I^k$ is 0 and the last coordinate of $I^{k+1}$ is 1. We call this a \emph{marked path} from $I$ to $I'$. Given such a marked path, we define a map $$G_{I^2<\dots<\wh{I^k}<\dots<I^n}: \cf(Y_{\alpha,\Lambda(I^2)})\rightarrow \cf(Y_{\alpha,\Lambda(I_n)})$$ by $$G_{I^2<\dots<\wh{I^k}<\dots<I^n}(\bx) = f^{J_{n,k}}_{\alpha,\Lambda(I^2),\dots,\Lambda(I^n)}(\bx\otimes\htheta_{2,3}\otimes\dots\otimes\htheta_{n-1,n}),$$ where $\htheta_{i,j}$ is the unique intersection point in $ \T_{\Lambda(I^i)} \cap \T_{\Lambda(I^{j})}$ with highest Maslov grading. Let $G_{I,I'}$ be the sum over marked paths from $I$ to $I'$, $$G_{I,I'}=\sum_{I= I^2<\dots<\wh{I^k}<\dots<I^n = I'} G_{I^2<\dots<\wh{I^k}<\dots<I^n},$$ and define $G$ to be the sum $$G = \sum G_{I,I'}.$$ Note that $G_{I,I'}$ is a map from $\cf(Y_{\alpha,\eta(\wt I)})$ to $\cf(Y_{\alpha,\lambda(\wt I')})$; as such, $G$ is a map from $X$ to $X'$.

\begin{proposition}
\label{prop:chain}
$G:(X,D)\rightarrow(X',D')$ is a chain map; that is, $G\circ D = D'\circ G$.
\end{proposition}

\begin{proof}[Proof of Proposition \ref{prop:chain}]
Fix intersection points $\bx \in \T_{\alpha}\cap\T_{\eta(\wt I)}$ and $\by \in \T_{\alpha}\cap\T_{\lambda(\wt{I'})}$ with $I\leq I'$ and $I_{m+1} = 0$ and $I'_{m+1} = 1$. Our goal is to show that the coefficient of $\by$ in $G\circ D(\bx)+D'\circ G(\bx)$ is zero mod 2. Let $S$ denote the set of marked paths from $I$ to $I'$, and fix some $p\in S$ given by $I=I^2<\dots<\wh{I^k}<\dots<I^n=I'$. Let $\pi_2(\bx,\by;p)$ denote the set of $\phi\in\pi_2(\bx,\htheta_{2,3},\dots,\htheta_{n-1,n},\by)$ with $n_z(\phi)=0$ and $\mu(\phi)=1$, and consider the 1-dimensional moduli spaces $\Mh_{J_{n,k}}(\phi)$ for $\phi\in\pi_2(\bx,\by;p)$. The ends of these moduli spaces consist of maps of degenerate or ``broken" $n$-gons. In particular, the compatibility conditions on the $J_{n,k}$ ensure that the ends of $\Mh_{J_{n,k}}(\phi)$ can be identified with the disjoint union $$X^1_{p,\phi}\,\,\coprod\, \,X^2_{p,\phi}\,\,\coprod \,\,X^3_{p,\phi}\,\,\coprod\,\, X^4_{p,\phi}\,\,\coprod \,\,X^5_{p,\phi},$$ where 

\begin{eqnarray*}
X^1_{p,\phi}&=&\coprod_{2\leq j\leq k}\,\,\,\coprod_{\bx\in\T_{\alpha}\cap\T_{\Lambda(I^j)}}\Mh_{J^{(0)}_j}(\phi_{\alpha,I^2,\dots,I^j;\bz})\times\Mh_{J_{n-j+2,k-j+2}}(\psi_{\alpha,I^j,\dots,I^n;\bz})\\
X^2_{p,\phi}&=&\coprod_{k<j\leq n}\,\,\,\coprod_{\bz\in\T_{\alpha}\cap\T_{\Lambda(I^j)}}\Mh_{J_{j,k}}(\phi_{\alpha,I^2,\dots,I^j;\bz})\times\Mh_{J^{(1)}_{n-j+2}}(\psi_{\alpha,I^j,\dots,I^n;\bz})\\
X^3_{p,\phi}&=&\coprod_{2\leq i<j \leq k}\,\,\,\coprod_{\bz\in\T_{\Lambda(I^i)}\cap\T_{\Lambda(I^j)}}\Mh_{J^{(0)}_{j-i+1}}(\phi_{I^i,\dots,I^j;\bz})\times\Mh_{J_{n-j+1+1,k-j+i+1}}(\psi_{\alpha,I^2,\dots,I^i,I^j,\dots,I^n;\bz})\\
X^4_{p,\phi}&=&\coprod_{2\leq i<k<j\leq n}\,\,\,\coprod_{\bz\in\T_{\Lambda(I^i)}\cap\T_{\Lambda(I^j)}}\Mh_{J_{j-i+1,k-i+1}}(\phi_{I^i,\dots,I^j;\bz})\times\Mh_{J_{n-j+i+1,i}}(\psi_{\alpha,I^2,\dots,I^i,I^j,\dots,I^n;\bz})\\
X^5_{p,\phi}&=&\coprod_{k\leq i<j\leq n}\,\,\,\coprod_{\bz\in\T_{\Lambda(I^i)}\cap\T_{\Lambda(I^j)}}\Mh_{J^{(1)}_{j-i+1}}(\phi_{I^i,\dots,I^j;\bz})\times\Mh_{J_{n-j+i+1,k}}(\psi_{\alpha,I^2,\dots,I^i,I^j,\dots,I^n;\bz}).
\end{eqnarray*}

In $X^1_{p,\phi}$ and $X^2_{p,\phi}$, the pairs $\phi_{\alpha, I^2,\dots,I^j;\bz}$ and $\psi_{\alpha,I^j,\dots,I^n;\bz}$ range over homotopy classes with $n_z = \mu = 0$ in $\pi_2(\bx,\htheta_{2,3},\dots,\htheta_{j-1,j},\bz)$ and $\pi_2(\bz,\htheta_{j,j+1},\dots,\htheta_{n-1,n},\by)$, respectively, such that $\phi$ is homotopic to the concatenation $\phi_{\alpha, I^2,\dots,I^j;\bz}* \psi_{\alpha, I^j,\dots,I^n;\bz}$. Similarly, the pairs $\phi_{I^i,\dots,I^j;\bz}$ and $\psi_{\alpha,I^2,\dots,I^i,I^j,\dots,I^n;\bz}$ in the unions $X^3_{p,\phi}$, $X^4_{p,\phi}$ and $X^5_{p,\phi}$ range over classes with $n_z = \mu = 0$ in $\pi_2(\htheta_{i,i+1},\dots,\htheta_{j-1,j},\bz)$ and $\pi_2(\bx,\htheta_{2,3},\dots,\htheta_{i-1,i},\bz,\htheta_{j,j+1}\dots,\htheta_{n-1,n})$, such that $\phi$ is homotopic to $\phi_{I^i,\dots,I^j;\bz}*\psi_{\alpha,I^2,\dots,I^i,I^j,\dots,I^n;\bz}$.

Observe that the number of points in the union $$\coprod_{\phi\in\pi_2(\bx,\by;p)} X^1_{p,\phi}$$ is precisely the coefficient of $\by$ in the contribution to $G\circ D(\bx)$ coming from the compositions $$G_{I^j<\dots<\wh{I^k}<\dots<I^n}\circ D_{\wt{I^2}<\dots<\wt{I^j}}$$ for $2\leq j\leq k$. Likewise, the number of points in the union $$\coprod_{\phi\in\pi_2(\bx,\by;p)} X^2_{p,\phi}$$ is the coefficient of $\by$ in the contribution to $D'\circ G(\bx)$ coming from the compositions $$D'_{\wt{I^j}<\dots<\wt {I^n}}\circ G_{I^2<\dots<\wh{I^k}<\dots<I^j}$$ for $k<j\leq n$. Therefore, the coefficient of $\by$ in $G\circ D(\bx)+D'\circ G(\bx)$ is the number of points in 
\begin{equation}
\label{eqn:union2}
\coprod_{p\in S}\,\,\coprod_{\phi\in\pi_2(\bx,\by;p)}(X^1_{p,\phi}\,\coprod \,X^2_{p,\phi}).
\end{equation} 
We would like to show that this number is zero mod 2. It is enough to show that the number of points in \begin{equation}
\label{eqn:union3}
\coprod_{p\in S}\,\,\coprod_{\phi\in\pi_2(\bx,\by;p)}(X^3_{p,\phi}\,\coprod \,X^4_{p,\phi}\,\coprod \,X^5_{p,\phi})\end{equation} 
is even, since the points in (\ref{eqn:union3}) together with those in (\ref{eqn:union2}) correspond precisely to the ends of the union of 1-dimensional moduli spaces $$\coprod_{p\in S}\,\,\coprod_{\phi\in\pi_2(\bx,\by;p)}\Mh_{J_{n,k}}(\phi),$$ and these ends are even in number. 

Consider the union $X^3_{p,\phi}.$ For $j-i=1$, the number of points corresponding to the term $\Mh_{J^{(0)}_{2}}(\phi_{I^i,I^{i+1};\bz})$ is even since these are the terms which appear as coefficients in the boundary $\partial_{J^{(0)}_{2}}(\htheta_{i,i+1})$, and $\htheta_{i,i+1}$ is a cycle. For $j-i>2 $, the set of classes in $\pi_2(\htheta_{i,i+1},\dots,\htheta_{j-1,j},\bz)$ with $n_z=\mu=0$ is empty by the same argument as is used in \cite[Lemma 4.5]{osz12}. Therefore, the only terms which contribute mod 2 to the number of points in $X^3_{p,\phi}$ are those in which $j-i=2$. 

For $j-i=2$, the set of classes in $\pi_2(\htheta_{i,i+1},\htheta_{i+1,i+2},\bz)$ with $n_z=\mu=0$ contains a single element as long as $\bz = \htheta_{i,i+2}$, and is empty otherwise. Let us denote this element by $\phi_{I^i,I^{i+1}, I^{i+2};\htheta_{i,i+2}}$. When $J^{(0)}_{3}$ is sufficiently close to the constant almost-complex structure $Sym^g(\mathfrak{j})$, there is exactly one pseudo-holomorphic representative of $\phi_{I^i,I^{i+1}, I^{i+2};\htheta_{i,i+2}}$; again, see \cite[Lemma 4.5]{osz12}. Since the map induced on homology by $f^{J^{(0)}_{3}}_{\Lambda(I^i),\Lambda(I^{i+1}),\Lambda(I^{i+2})}$ is independent of the the analytic data which goes into its construction \cite[Lemma 10.19]{lip}, the number of points in $\Mh_{J^{(0)}_{3}}(\phi_{I^i,I^{i+1}, I^{i+2};\htheta_{i,i+2}})$ is 1 mod 2 for any generic $J^{(0)}_{3}$. So, mod 2, the number of points in $X^3_{p,\phi}$ is the same as the number of points in 
\begin{equation}
\label{eqn:mod} 
\coprod_{2\leq i< k}\Mh_{J_{n-1,k-1}}(\psi_{\alpha,I^2,\dots,I^i,I^{i+2},\dots,I^n;\htheta_{i,i+2}}).
\end{equation} 
Note that each term in the disjoint union above also appears as a term in the mod 2 count of the number of points in $X^3_{p',\phi}$ for exactly one other $p'$ since, for each $2\leq i< k$, there is exactly one other marked path $p'\in S$ given by $$I=I^2<\dots<I^i<I''<I^{i+2}<\dots<\wh{I^k}<\dots<I^n = I'.$$ Therefore, the number of points in $$\coprod_{p\in S}\,\,\coprod_{\phi\in\pi_2(\bx,\by;p)}X^3_{p,\phi}$$ is even. Similar statements hold for $X^4_{p,\phi}$ and $X^5_{p,\phi}$ by virtually identical arguments. Hence, the number of points in (\ref{eqn:union3}) is even, and we are done.

\end{proof}

\begin{proposition}
\label{prop:e1}
The map $G$ induces an isomorphism from $E^1(X)$ to $E^1(X')$.
\end{proposition}

\begin{proof}[Proof of Proposition \ref{prop:e1}]
The map from $E^1(X)$ to $E^1(X')$ induced by $G$ is the portion of the adjusted map $G^{(1)}$ which preserves the grading $h'$. It is therefore the sum of the maps $$(G_{I,I'})_*:\hf(Y_{\alpha,\eta(\wt I)})\rightarrow \hf(Y_{\alpha, \lambda(\wt{I'})}),$$ over all $I<I'\in \{0,1\}^{m+1}$ for which there is a marked path of the form $I=\wh I^2<I^3=I'$. Equivalently, the map induced by $G$ is the sum, over $K\in \{0,1\}^m$, of the maps \begin{equation}\label{eqn:map}F^{J_{3,2}}_{\alpha,\eta(K),\lambda(K)}(-\otimes\htheta_K): \hf(Y_{\alpha,\eta(K)})\rightarrow \hf(Y_{\alpha, \lambda(K)}),\end{equation} where $\htheta_K$ is the unique intersection point in $\T_{\eta(K)}\cap\T_{\lambda(K)}$ with highest Maslov grading, and $F^{J_{3,2}}_{\alpha,\eta(K),\lambda(K)}(-\otimes\htheta_K)$ is the map on homology induced by $f^{J_{3,2}}_{\alpha,\eta(K),\lambda(K)}(-\otimes\htheta_K)$. If $J_{3,2}$ is sufficiently close to the constant almost-complex structure $Sym^g(\mathfrak{j})$, then the maps in (\ref{eqn:map}) are isomorphisms by an area filtration argument \cite[Section 9]{osz8}. Therefore, these maps are isomorphisms for any generic $J_{3,2}$ (again, by \cite[Lemma 10.19]{lip}).

\end{proof}

It follows from Lemma \ref{lem:filt} that $(X,D)$ is filtered chain homotopy equivalent to $(X',D')$. In addition, the complex $(X',D')$ is filtered chain homotopy equivalent to the complex associated to the multi-diagram $H$ and the analytic data $\mathfrak{j}$ and $\{J^{(1)}_n\}_{n\geq 2}$, by Theorem \ref{thm:multi}. In other words, we have shown that for a multi-diagram $H$ and a fixed $\mathfrak{j}$, the filtered chain homotopy type of the complex associated to $H$ does not depend on the choice of $\{J_n\}_{n\geq 2}$. It follows that the filtered chain homotopy type of this complex is also independent of $\mathfrak{j}$, as in \cite[Theorem 6.1]{osz8}. 
This completes the proof of Theorem \ref{thm:analytic}.

\qed

\section{The spectral sequence from $\kh$ to $\hf$}
\label{sec:review}
In this section, we describe how a specialization of the link surgeries spectral sequence gives rise to a spectral sequence from the reduced Khovanov homology of a link to the Heegaard Floer homology of its branched double cover, following \cite{osz12}.

Let $L$ be a planar diagram for an oriented link in $S^3$, and label the crossings of $L$ from $1$ to $m$. For $I = (I_1,\dots,I_m) \in \{\infty,0,1\}^m$, let $L_I$ be the planar diagram obtained from $L$ by taking the $I_j$-resolution of the $j$th crossing for each $j \in \{1,\dots, m\}$.

 \begin{figure}[!htbp]
 \labellist 
\small\hair 2pt 
\pinlabel $\infty$ at 29 2 
\pinlabel $0$ at 101 2
\pinlabel $1$ at 178 2 
\endlabellist 
\begin{center}
\includegraphics[width=7cm]{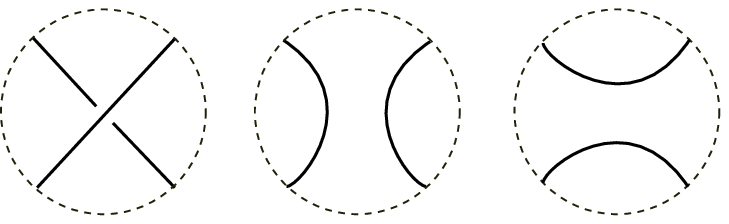}
\caption{\quad The $\infty$-, $0$- and $1$- resolutions of a crossing.}
\label{fig:Res}
\end{center}
\end{figure}

Let $a_j$ denote the dashed arc in the local picture near the $j$th crossing of $L$ shown in Figure \ref{fig:arc}. The arc $a_j$ lifts to a closed curve $\alpha_j$ in the branched double cover $-\Sigma(L)$. For $I \in \{\infty,0,1\}^m$, $-\Sigma(L_I)$ is obtained from $-\Sigma(L)$ by performing $I_j$-surgery on $\alpha_j$ with respect to some fixed framing, for each $j\in \{1,\dots,m\}$. 

 \begin{figure}[!htbp]
 \labellist 
\small\hair 2pt 
\pinlabel $j$ at 37 30 
\pinlabel $a_j$ at 14 29 
\endlabellist 
\begin{center}
\includegraphics[width=2.5cm]{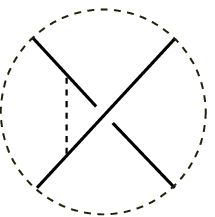}
\caption{}
\label{fig:arc}
\end{center}
\end{figure}

Let $(X,D)$ be the complex associated to some multi-diagram compatible with the framed link $$\mathbb{L}_L = \alpha_1\cup\dots\cup\alpha_m \subset -\Sigma(L)$$ and some choice of analytic data. Recall that $$X=\bigoplus_{I\in \{0,1\}^m}\cf(-\Sigma(L_I)), $$ and $D$ is the sum of maps $$D_{I,I'}:\cf(-\Sigma(L_I))\rightarrow \cf(-\Sigma(L_{I'})),$$ over all pairs $I \leq I'$ in $\{0,1\}^m$. In this context, Theorem \ref{thm:conv} says the following.

\begin{theorem}
\label{thm:homology}
The homology $H_*(X,D)$ is isomorphic to $\hf(-\Sigma(L))$.
\end{theorem}

There is a grading on $X$ defined, for $x \in \cf(-\Sigma(L_I))$, by $h(x) = I_1+\cdots + I_n-n_-(L)$, where $n_-(L)$ is the number of negative crossings in $L$. Note that $h$ is just a shift of the grading $h'$ defined in Section \ref{sec:lsss}. As such, this $h$-grading gives rise to an ``$h$-filtration" of the complex $(X,D),$ which, in turn, gives rise to the link surgery spectral sequence associated to $\bL_L$. 

Let $E^k(L)$ denote the $E^k$ term of this spectral sequence for $k\geq1$. Though the complex $(X,D)$ depends on a choice of multi-diagram compatible with $\bL_L$ and a choice of analytic data, the $h$-graded vector space $E^k(L)$ depends only on the diagram $L$, by Theorems \ref{thm:multi} and \ref{thm:analytic}. In light of this fact, we will often use the phrase ``the complex associated to a planar diagram $L$" to refer to the complex associated to any multi-diagram compatible with $\mathbb{L}_L$ and any choice of analytic data. The primary goal of this article is to show that, in fact, $E^k(L)$ depends only on the topological link type of $L$.

Recall that the portion $D^0$ of $D$ which does not shift the $h$-grading is the sum of the standard Heegaard Floer boundary maps $$D_{I,I}:\cf(-\Sigma(L_I)) \rightarrow \cf(-\Sigma(L_I)).$$ Therefore, $$E^1(L)\cong \bigoplus_{I\in \{0,1\}^m}\hf(-\Sigma(L_I)). $$ If $I'$ is an immediate successor of $I$, then $-\Sigma(L_{I'})$ is obtained from $-\Sigma(L_{I})$ by performing $(-1)$-surgery on a meridian of $\alpha_k$, and $$(D_{I,I'})_*:\hf(-\Sigma(L_{I}))\rightarrow \hf(-\Sigma(L_{I'}))$$ is the map induced by the corresponding 2-handle cobordism. By construction, the differential $D^1$ on $E^1(L)$ is the sum of the maps $(D_{I,I'})_*$, over all pairs $I,I'$ for which $I'$ is an immediate successor of $I$. 

\begin{theorem}[{\rm \cite[Theorem 6.3]{osz12}}]
\label{thm:kh}
The complex $(E^1(L),D^1)$ is isomorphic to the complex $(\ckh(L),d)$ for the reduced Khovanov homology of  $L$. In particular, $E^2(L) \cong \kh(L)$.
\end{theorem}

We refer to $h$ as the ``homological grading" on $X$ since the induced $h$-grading on $E^1(L)$ agrees with the homological grading on $\ckh(L)$ via the isomorphism above. 

Recall that when $s$ is a torsion $Spin^c$ structure, the chain complex $\cf(-\Sigma(L_I),s)$ is equipped with an absolute Maslov grading $m$ which takes values in $\mathbb{Q}$ \cite{osz6}. We can use this Maslov grading to define a grading $q$ on a portion of $X$, as follows. Suppose that $s$ is torsion, and let $x$ be an element of  $\cf(-\Sigma(L_I),s)$ which is homogeneous with respect to $m$. We define $q(x) = 2m(x) + h(x) + n_+(L) - n_-(L)$, where $n_+(L)$ is the number of positive crossings in $L$. Since each $-\Sigma(L_I)$ is a connected sum of $S^1\times S^2$s, and the Heegaard Floer homology of such a connected sum is supported in a unique torsion $Spin^c$ structure, there is an induced $q$-grading on all of $E^1(L)$. Moreover, this induced grading agrees with the quantum grading on $\ckh(L)$. As such, we shall refer to $q$ as the partial ``quantum grading" on $X$. 

It is not clear that there is always a well-defined quantum grading on $E^k(L)$ for $k>2$, though this is sometimes the case. For example, suppose that $(X,D)$ is the complex associated to some multi-diagram compatible with $\mathbb{L}_L$ and some choice of analytic data. If the partial quantum grading on $X$ gives rise to a well-defined quantum grading on the higher terms of the induced spectral sequence, then there is a well-defined quantum grading on the higher terms of the spectral sequence induced by any other compatible multi-diagram and choice of analytic data. Moreover, these quantum gradings agree. 
This follows from the proofs of independence of multi-diagram and analytic data: the isomorphisms between $E^1$ terms preserve quantum grading; therefore, so do the induced isomorphisms between higher terms. Furthermore, it follows from Section \ref{sec:reid} that this quantum grading, if it exists, does not depend on the choice of planar diagram; that is, it is a link invariant. 

As a last remark, we note that Lemma \ref{lem:connect} implies the following corollary.

\begin{corollary}
\label{cor:connect}
For links $L_1$ and $L_2$, $E^k(L_1\#L_2) \cong E^k(L_1)\otimes E^k(L_2)$.
\end{corollary}

\section{Invariance under the Reidemeister moves}
\label{sec:reid}

\begin{theorem}
\label{thm:invcE}
If $L$ and $L'$ are two planar diagrams for a link, then $E^k(L)$ is isomorphic to $E^k(L')$ as an $h$-graded vector space for all $k \geq 2$. 
\end{theorem}

It suffices to check Theorem \ref{thm:invcE} for diagrams $L$ and $L'$ which differ by a Reidemeister move; we do this in the next three subsections. 

\subsection{Reidemeister I}
\label{ssec:RI}
Let $L^+$ be the diagram obtained from $L$ by adding a positive crossing via a Reidemeister I move. Let $(X,D)$ be the complex associated to a multi-diagram $H$ compatible with $\mathbb{L}_{L^+}$. Label the crossings of $L^+$ by $1,\dots,m+1$ so that crossing $m+1$ corresponds to the positive crossing introduced by the Reidemeister I move. As in \cite{osz12}, the multi-diagram $H$ actually gives rise to a larger complex $(\ol X, \ol D)$, where $$\ol X = \bigoplus_{I\in \{0,1\}^m\times \{\infty,0,1\}} \cf(-\Sigma(L^+_I)),$$ and $\ol D$ is a sum of maps $$D_{I,I'}:\cf(-\Sigma(L^+_I)) \rightarrow \cf(-\Sigma(L^+_{I'}))$$ over pairs $I\leq I'$ in $\{0,1\}^m\times \{\infty,0,1\}$.

For $j \in \{\infty,0,1\}$, let $(X_{*j},D_{*j})$ be the complex for which $$X_{*j} = \bigoplus_{I \in \{0,1\}^m \times \{j\}} \cf(-\Sigma(L^+_I)),$$ and $D_{*j}$ is the sum of the maps $D_{I,I'}$ over all pairs $I \leq I'$ in $\{0,1\}^m \times \{j\}$. For $j<j'$ in $\{\infty, 0, 1\}$, let $$F_{j,j'}: X_{*j}\rightarrow X_{*j'}$$ be the sum of the maps $D_{I,I'}$ over pairs $I \in \{0,1\}^m\times \{j\}$, $I' \in \{0,1\}^m\times \{j'\}$ with $I< I'$. Then, $(X,D)$ is the mapping cone of $$F_{0,1}: (X_{*0},D_{*0})\rightarrow (X_{*1},D_{*1}),$$ and $$F_{\infty,0}\oplus F_{\infty,1}: (X_{*\infty}, D_{*\infty}) \rightarrow (X,D)$$ is an $h$-filtered chain map, where the $h$-grading on $X_{*\infty}$ is defined by $h(x) =(I_1 + \dots + I_m) - n_-(L)$ for $x \in \cf(-\Sigma(L^+_I))$ and  $I\in \{0,1\}^m \times \{\infty\}$. Note that the sub-diagram of $H$ used to define the complex $(X_{*\infty}, D_{*\infty})$ is compatible with the framed link $\mathbb{L}_L$. We may therefore think of $(X_{*\infty}, D_{*\infty})$ as the graded complex associated to $L$. 

 \begin{figure}[!htbp]

 \labellist 
\tiny\hair 2pt 
\pinlabel $*\infty$ at 45 16 
\pinlabel $*0$ at 222 16 
\pinlabel $*1$ at 395 16 
\pinlabel $F_{\infty,0}$ at 131 85 
\pinlabel $F_{\infty,1}$ at 213 175 
\pinlabel $F_{0,1}$ at 305 85

\endlabellist 
\begin{center}
\includegraphics[width=7.8cm]{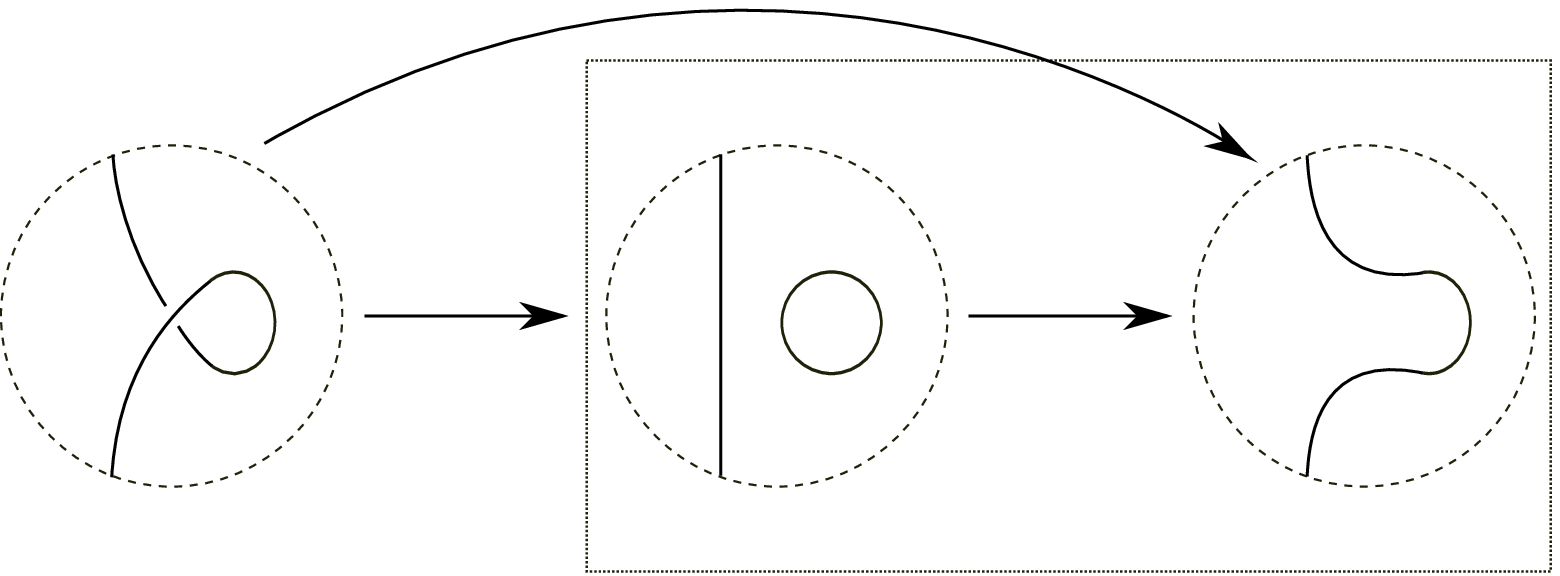}
\caption{\quad This is the complex $(\ol X, \ol D).$ The tangle labeled by $*j$ is meant to represent the complex $(X_{*j},D_{*j})$. The subcomplex surrounded by the box is precisely $(X, D)$.}
\label{fig:tcomplex}
\end{center}
\end{figure}

First, cancel the components of the differentials $D_{*j}$ which do not change the $h$-grading, and let $F_{j,j'}^{(1)}$ denote the adjusted maps. Observe that $$X_{*j}^{(1)} \cong \bigoplus_{I\in \{0,1\}^m \times \{j\}}\hf(-\Sigma(L^+_I)).$$ For $j \in \{\infty,0,1\}$, the spectral sequence differential $D_{*j}^1$ is the sum of the components of $D_{*j}^{(1)}$ which increase the $h$-grading by 1, as explained in Section \ref{sec:ss}. Let $A$ be the sum of the components of $F^{(1)}_{\infty,0}$ which do not change the $h$-grading, and let $B$ be the sum of the components of $F_{0,1}^{(1)}$ which increase the $h$-grading by 1. Note that $A$ is the map from $E^1(L)$ to $E^1(L^+)$ induced by $F_{\infty,0}\oplus F_{\infty,1}$.

For each $I\in \{0,1\}^m$, there is a surgery exact triangle \cite{osz12} \begin{displaymath}
\xymatrix{
  \hf(-\Sigma(L^+_{I\times \{\infty\}}))  \ar[r]^{A_I}&  \hf(-\Sigma(L^+_{I\times \{0\}})) \ar[d]^{B_I} \\
 &\hf(-\Sigma(L^+_{I\times \{1\}})),\ar[ul]^{C_I} \\}
\end{displaymath} where $A_I$ is the map induced by the 2-handle cobordism corresponding to $0$-surgery on the curve $\alpha_{m+1}$ (defined in Section \ref{sec:review}), viewed as an unknot in $-\Sigma(L^+_{I\times \{\infty\}})$.
The maps $C_I$ are all $0$ since $$\text{rk}\,\hf(-\Sigma(L^+_{I\times \{0\}})) = \text{rk}\,\hf(-\Sigma(L^+_{I\times \{\infty\}})) + \text{rk}\,\hf(-\Sigma(L^+_{I\times \{1\}})).$$ Moreover, $A$ and $B$ are the sums over $I\in \{0,1\}^m$ of the maps $A_I$ and $B_I$, respectively. It follows that the complex \begin{displaymath}
\xymatrix{
(X_{*\infty}^{(1)},D_{*\infty}^1)  \ar[r]^{A}& (X_{*0}^{(1)},D_{*0}^1) \ar[r]^{B}& (X_{*1}^{(1)},D_{*1}^1)   }
\end{displaymath} is acyclic. Equivalently, $A$ induces an isomorphism from $H_*(X_{*\infty}^{(1)},D_{*\infty}^1)=E^2(L),$ to the homology of the mapping cone of $B$, which is $E^2(L^+).$ Therefore, Lemma \ref{lem:filt} implies that $F_{\infty,0}\oplus F_{\infty,1}$ induces a graded isomorphism from $E^k(L)$ to $E^k(L^+)$ for all $k \geq 2$.

 \begin{figure}[!htbp]

 \labellist 
\tiny\hair 2pt 
\pinlabel $*0$ at 45 16 
\pinlabel $*1$ at 222 16 
\pinlabel $*\infty$ at 395 16 
\pinlabel $F_{\infty,0}$ at 135 89 
\pinlabel $F_{\infty,1}$ at 217 176 
\pinlabel $F_{0,1}$ at 310 89

\endlabellist 
\begin{center}
\includegraphics[width=7.8cm]{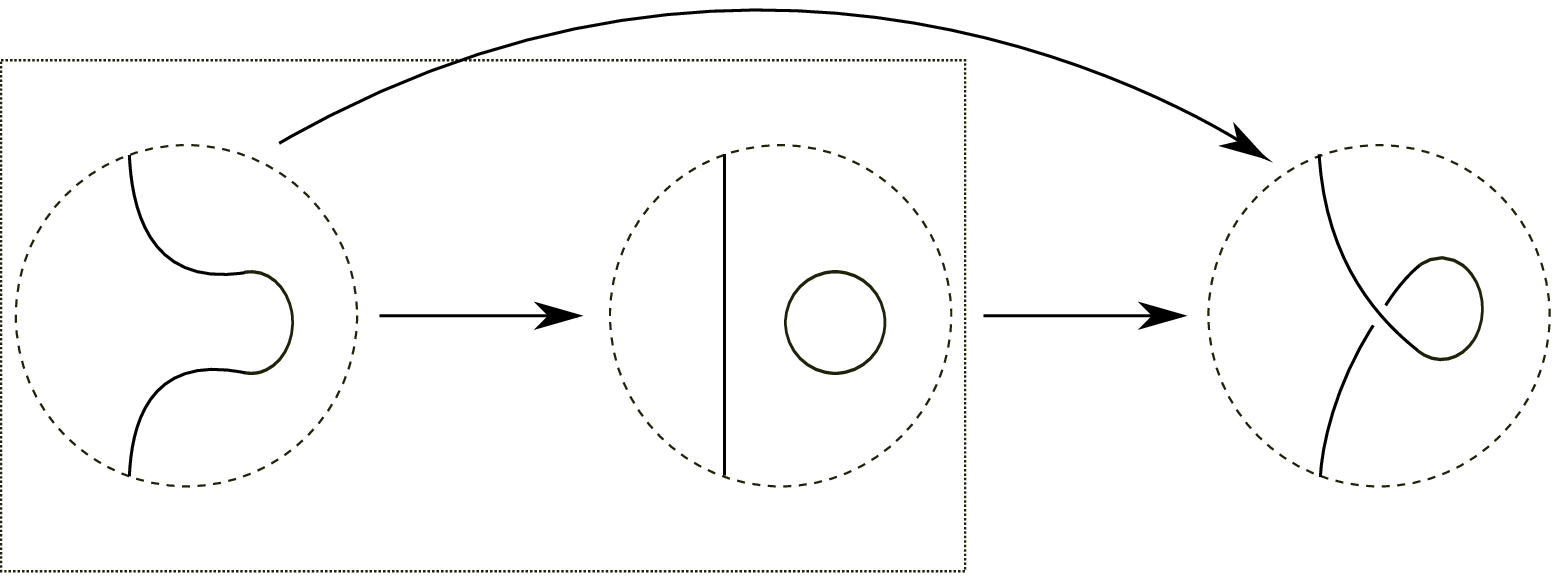}
\caption{\quad The complex $(\widehat X, \widehat D)$. The tangle labeled by $*j$ is meant to represent the complex $(X_{*j},D_{*j})$. The quotient complex surrounded by the box is precisely the complex $(X, D)$.}
\label{fig:t2complex}
\end{center}
\end{figure}

The proof of invariance under a Reidemeister I move which introduces a \emph{negative} crossing is more or less the same. We omit the details, though Figure \ref{fig:t2complex} gives a schematic depiction of the filtered chain map $$F_{\infty,1} \oplus F_{0,1}: (X,D) \rightarrow (X_{*\infty},D_{*\infty})$$ which induces a graded isomorphism from $E^k(L^-)$ to $E^k(L)$ for all $k \geq 2$. In this setting, $(X,D)$ is the complex associated to  the diagram $L^-$ obtained from $L$ via a negative Reidemeister I move. Everything else is defined similarly; as before, we may think of $(X_{*\infty},D_{*\infty})$ as the complex associated to $L$.

\subsection{Reidemeister II}
\label{ssec:RII}
Suppose that $\wt L$ is the diagram obtained from $L$ via a Reidemeister II move. Label the crossings of $\wt L$ by $1,\dots,m+2$ so that crossings $m+1$ and $m+2$ correspond to the top and bottom crossings, respectively, introduced by the Reidemeister II move shown in Figure \ref{fig:RII}. Let $( X,D)$ be the complex associated to a multi-diagram compatible with the framed link $\mathbb{L}_{\wt L}$. For $j\in \{0,1\}^2$, denote by $I_{*j}$ the subset of vectors in $\{0,1\}^{m+2}$ which end with the string specified by $j$. 

\begin{figure}[!htbp]
\labellist 
\small\hair 2pt 
\pinlabel $\wt L$ at 251 4
\pinlabel $L$ at 64 4
\endlabellist 
\begin{center}
\includegraphics[width=5.2cm]{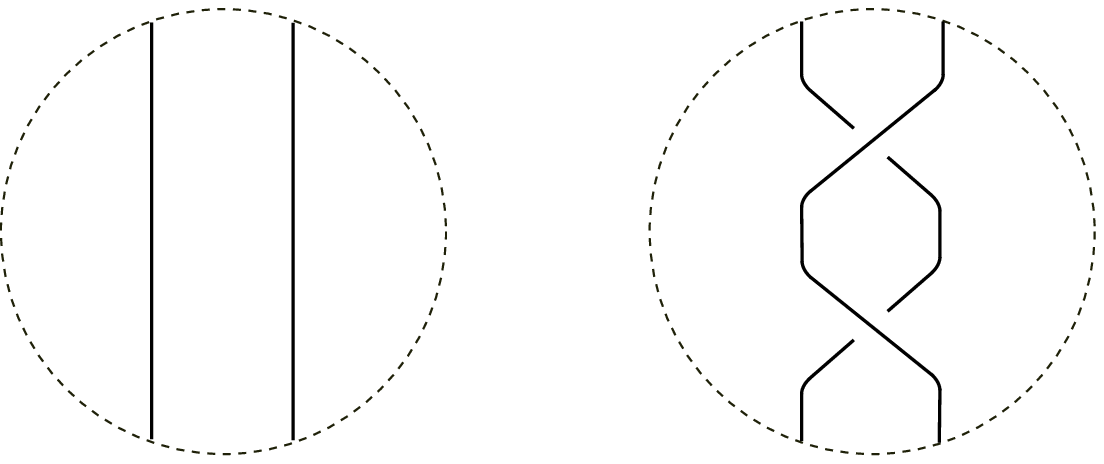}
\caption{}
\label{fig:RII}
\end{center}
\end{figure}

Define $$X_{*j} = \bigoplus_{I\in I_{*j}} \cf(-\Sigma(\wt L_I)),$$ and let $D_{*j}$ be the sum the maps $D_{I,I'}$ over all pairs $I\leq I'$ in $I_{*j}$. Then $$X = \bigoplus_{j\in\{0,1\}^2} X_{*j},$$ and $D$ is the sum of the differentials $D_{*j}$ together with the maps $$F_{j,j'}:X_{*j} \rightarrow X_{*j'}$$ for $j < j',$ where $F_{j,j'}$ is itself the sum of the maps $D_{I,I'}$ over all pairs $I \in I_{*j}$, $I' \in I_{*j'}$ with $I< I'$. Note that the sub-diagram used to define the complex $(X_{*j}, D_{*j})$ is compatible with the framed link $\mathbb{L}_{\wt L_{*j}}$, where $\wt L_{*j}$ is the planar diagram obtained from $\wt L$ by taking the $j_1$-resolution of crossing $m+1$ and the $j_2$-resolution of crossing $m+2$. In particular, we may think of $(X_{*01}, D_{*01})$ as the graded complex associated to the diagram $L$. See Figure \ref{fig:R2} for a more easy-to-digest depiction of the complex $(X,D)$.


\begin{figure}[!htbp]
\labellist 
\tiny\hair 2pt 
\pinlabel $*01$ at 48 218 
\pinlabel $*00$ at 48 7
\pinlabel $*11$ at 261 218 
\pinlabel $*10$ at 261 7
\pinlabel $C$ at 288 66
\pinlabel $F_{00,01}$ at 152 52
\pinlabel $F_{01,11}$ at 152 290
\pinlabel $F_{00,11}$ at 140 185
\pinlabel $F_{00,01}$ at 24 160
\pinlabel $F_{10,11}$ at 286 160

\endlabellist 
\begin{center}
\includegraphics[width=6cm]{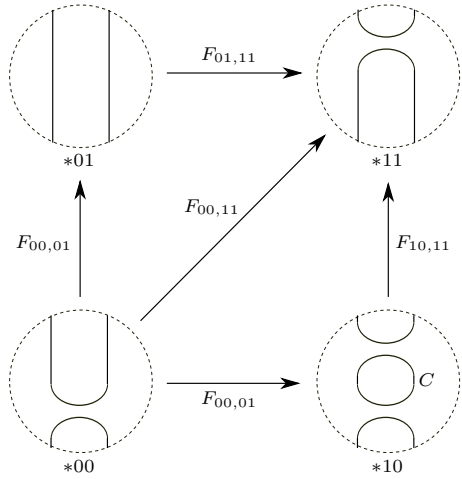}
\caption{\quad In this diagram, the tangle labeled by $*j$ is meant to represent the complex $(X_{*j}, D_{*j})$. We have labeled the circle $C$ in the tangle at the lower right for later reference.}
\label{fig:R2}
\end{center}

\end{figure}

First, cancel all components of $D$ which do not change the $h$-grading. The resulting complex is $(X^{(1)}, D^{(1)})$, where $$X^{(1)} = \bigoplus_{I \in \{0,1\}^2} X^{(1)}_{*j},$$ and $D^{(1)}$ is the sum of the differentials $D^{(1)}_{*j}$ and the adjusted maps $F^{(1)}_{j,j'}$. 
Note that $$X_{*j}^{(1)} \cong \bigoplus_{I\in I_{*j}}\hf(-\Sigma(\wt L_I))\cong  \bigoplus_{I\in I_{*j}}\ckh(\wt L_I).$$ Denote by $A$ (resp. $B$) the sum of the components of $F_{00,10}^{(1)}$ (resp. $F_{10,11}^{(1)}$) which increase the $h$-grading by 1. 
By Theorem \ref{thm:kh}, and via the identification above, we may think of $A$ and $B$ as the maps $$A: \bigoplus_{I\in I_{*00}} \ckh(\wt L_I) \rightarrow \bigoplus_{I\in I_{*10}} \ckh(\wt L_I)$$ and $$B: \bigoplus_{I\in I_{*10}} \ckh(\wt L_I) \rightarrow \bigoplus_{I\in I_{*11}} \ckh(\wt L_I)$$ on the Khovanov chain complex induced by the corresponding link cobordisms. 

Recall that in reduced Khovanov homology, the vector space $\ckh(\wt L_I)$ is a quotient of $V^{\otimes |\wt L_I|}$, where $V$ is the free $\zzt$-module generated by $\bv_+$ and $\bv_-$ and $|\wt L_I|$ is the number of components of $\wt L_I$; we think of each factor of $V$ as corresponding to a circle in the resolution $\wt L_I$. 
Let $X_+$ (resp. $X_-$) be the vector subspace of $X_{*10}^{(1)}$ generated by those elements for which $\bv_+$ (resp. $\bv_-$) is assigned to the circle $C$ in Figure \ref{fig:R2}. Then \begin{equation}\label{eqn:summ}X_{*10}^{(1)} = X_+ \oplus X_-.\end{equation} Note that $A$ is a sum of splitting maps while $B$ is a sum of merging maps. Therefore, the component $A_-$ of $$A:X_{*00}^{(1)} \rightarrow X_{*10}^{(1)}$$ which maps to the second summand of the decomposition in Equation (\ref{eqn:summ}) is an isomorphism; likewise, the restriction $B_+$ of $$B:X_{*10}^{(1)} \rightarrow X_{*11}^{(1)}$$ to the first summand is an isomorphism \cite{kh1}. See Figure \ref{fig:AB} for a pictorial depiction of the composition $B\circ A$.

\begin{figure}[!htbp]
\labellist 
\hair 2pt 
\pinlabel $\oplus$ at 218 124 
\tiny\pinlabel $A$ at 131 371 
\pinlabel $B$ at 301 371 
\tiny \pinlabel $*00$ at 49 301 
\pinlabel $*10$ at 218 301 
\pinlabel $*11$ at 388 301 
\pinlabel $*00$ at 49 68 
\pinlabel $X_+$ at 219 247 
\pinlabel $*11$ at 388 68 
\pinlabel $X_-$ at 219 -2 
\pinlabel $A_-$ at 138 102 
\pinlabel $\cong$ at 128 84 
\pinlabel $B_+$ at 309 168
\pinlabel $\cong$ at 299 151 
\pinlabel $\bv_+$ at 251 185 
\pinlabel $\bv_-$ at 251 56

\endlabellist 
\begin{center}
\includegraphics[width=7.5cm]{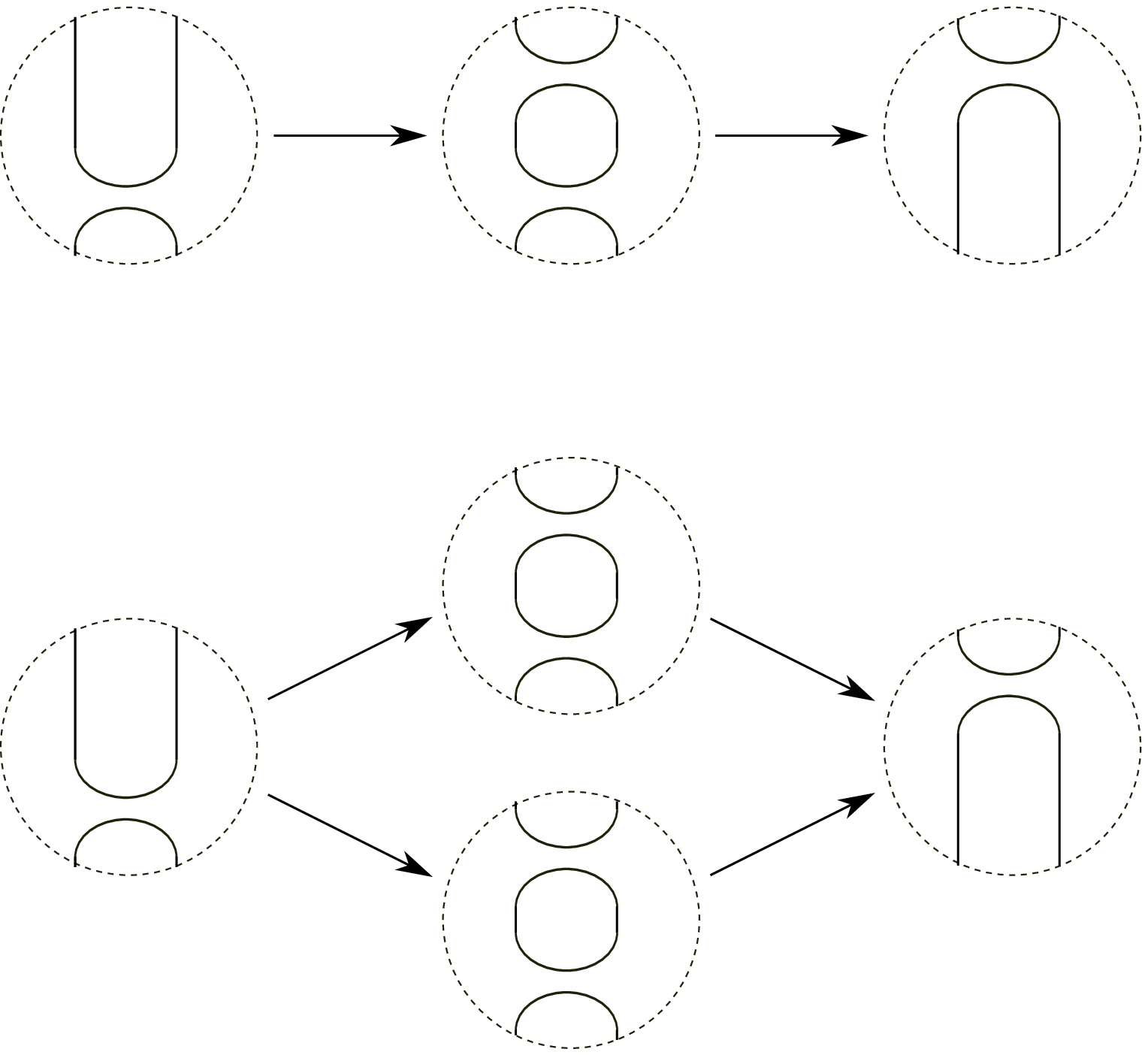}
\caption{\quad In this diagram, the tangle marked $*j$ is meant to represent the vector space $X^{(1)}_{*j}$. The figure on top shows the composition $B\circ A$. The bottom figure illustrates this composition with respect to the identification of $X_{*10}^{(1)}$ with $X_+ \oplus X_-$.}
\label{fig:AB}
\end{center}
\end{figure}

Therefore, after canceling all components $A_-$ and $B_+$, all that remains of $(X^{(1)},D^{(1)})$ is the complex $(X_{*01}^{(1)},D_{*01}^{(1)}).$ It follows that $(X^{(k)}, D^{(k)}) = (X_{*01}^{(k)},D_{*01}^{(k)})$ for all $k \geq 2$. In particular, $E^k(\wt L) \cong E^k(L)$ as graded vector spaces for all $k \geq 2$.

\subsection{Reidemeister III}
\label{ssec:RIII}
The proof of invariance under Reidemeister III moves is very similar in spirit to the proof for Reidemeister II. If $x$ and $y$ are the elementary generators of the braid group on 3 strands, then every Reidemeister III move corresponds to isolating a 3-stranded tangle in $L$ associated to the braid word $yxy$ (or $y^{-1}x^{-1}y^{-1}$), and replacing it with the tangle associated to $xyx$ (or $x^{-1}y^{-1}x^{-1}$) (although we are using braid notation, we are not concerned with the orientations on the strands). We can also perform a Reidemeister III move by isolating a trivial 3-tangle adjacent to the tangle $yxy$, and replacing it with the tangle $xy xy^{-1}x^{-1}y^{-1}$. The concatenation of these two tangles is the tangle $xyxy^{-1}x^{-1}y^{-1} yxy$, which is isotopic to the tangle $xyx$ via Reidemeister II moves: $$xyxy^{-1}x^{-1}y^{-1} yxy\, \,\sim \,\,xyxy^{-1}x^{-1} xy\, \,\sim \,\,xyxy^{-1} y\,\, \sim \,\,xyx$$ (the move from $y^{-1}x^{-1}y^{-1}$ to $x^{-1}y^{-1}x^{-1}$ can also be expressed in this way). Since $E^k(L)$ is invariant under Reidemeister II moves, invariance under Reidemeister III follows if we can show that $E^k(\wt L) \cong E^k(L)$, where $\wt L$ is the diagram obtained from $L$ by replacing a trivial 3-stranded tangle with the tangle associated to the word $xyxy^{-1}x^{-1}y^{-1}$ (see Figure \ref{fig:R3}). 

\begin{figure}[!htbp]
\labellist 
\small\hair 2pt 
\pinlabel $\wt L$ at 251 9 
\pinlabel $L$ at 64 9 
\endlabellist 
\begin{center}
\includegraphics[width=5.4cm]{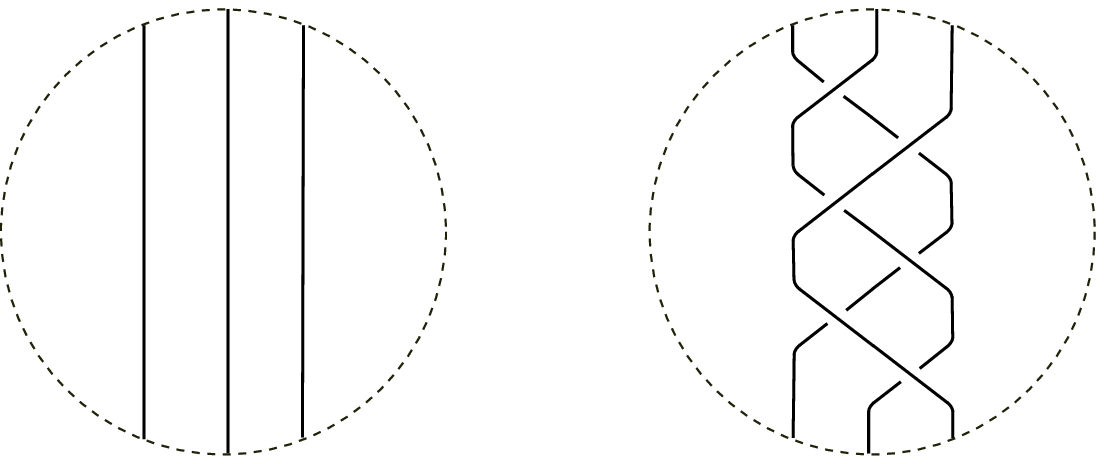}
\caption{}
\label{fig:R3}
\end{center}
\end{figure}

Label the crossings of $\wt L$ by $1,\dots,m+6$ so that crossings $m+1,\dots,m+6$ correspond to the 6 crossings (labeled from top to bottom) introduced by replacing the trivial 3-tangle with the tangle $xyx y^{-1}x^{-1}y^{-1}$ as shown in Figure \ref{fig:R3}. Let $(X,D)$ be the complex associated to a multi-diagram compatible with the framed link $\mathbb{L}_{\wt L}$. For $j\in \{0,1\}^6$, denote by $I_{*j}$ the subset of vectors in $\{0,1\}^{m+6}$ which end with the string specified by $j$. As before, define $$X_{*j} = \bigoplus_{I\in I_{*j}} \cf(-\Sigma(\wt L_I)),$$ and let $D_{*j}$ be the sum the maps $D_{I,I'}$ over all pairs $I\leq I'$ in $I_{*j}$. Then $$X = \bigoplus_{j\in\{0,1\}^6} X_{*j},$$ and $D$ is the sum of the differentials $D_{*j}$ together with the maps $$F_{j,j'}:X_{*j} \rightarrow X_{*j'}$$ for $j < j',$ where $F_{j,j'}$ is the sum of the maps $D_{I,I'}$ over all pairs $I \in I_{*j}$, $I' \in I_{*j'}$ with $I<I'$. We may think of $(X_{*000111}, D_{*000111})$ as the graded complex associated to the diagram $L$. 

Cancel all components of $D$ which do not change the $h$-grading. The resulting complex is $(X^{(1)}, D^{(1)})$, where $$X^{(1)} = \bigoplus_{j \in \{0,1\}^6} X^{(1)}_{*j},$$ and $D^{(1)}$ is the sum of the differentials $D^{(1)}_{*j}$ and the adjusted maps $F^{(1)}_{j,j'}$. As before, $$X_{*j}^{(1)} \cong \bigoplus_{I\in I_{*j}}\hf(-\Sigma(\wt L_I))\cong  \bigoplus_{I\in I_{*j}}\ckh(\wt L_I),$$ and the components of $D^{(1)}$ which increase the $h$-grading by 1 are precisely the maps on reduced Khovanov homology induced by the corresponding link cobordisms.

Our aim is to perform cancellation until all that remains of $(X^{(1)}, D^{(1)})$ is the complex $(X^{(1)}_{*000111}, D^{(1)}_{*000111}).$ The complex $(X^{(1)}, D^{(1)})$ may be thought of as a 6 dimensional hypercube whose vertices are the 64 complexes $(X^{(1)}_{*j}, D^{(1)}_{*j})$ for $j \in \{0,1\}^6$.  The edges of this cube represent the components of the maps $F^{(1)}_{j,j'}$ which increase the $h$-grading by 1; all cancellation will take place among these edge maps. To avoid drawing this 64 vertex hypercube, we perform some preliminary cancellations.

For each $s\in \{0,1\}^3$, consider the 3 dimensional face whose vertices are the 8 complexes $(X^{(1)}_{*ls}, D^{(1)}_{*ls})$, where $l$ ranges over $\{0,1\}^3$. The diagram on the left of Figure \ref{fig:web2a} depicts this face. In this figure, we have zoomed in on the portion of the link which changes as $l$ varies (note that this is the cube of resolutions for the tangle $xyx$).

\begin{figure}[!htbp]
\labellist 

\tiny 
\pinlabel $*000s$ at 40 112
\pinlabel $*010s$ at 160 112
\pinlabel $*101s$ at 286 112
\pinlabel $*111s$ at 407 112
\pinlabel $*001s$ at 160 232
\pinlabel $*100s$ at 160 -7
\pinlabel $*011s$ at 286 232
\pinlabel $*110s$ at 286 -7

\pinlabel $*000s$ at 577 112
\pinlabel $*010s$ at 679 182
\pinlabel $*100s$ at 679 41
\pinlabel $*011s$ at 802 182
\pinlabel $*110s$ at 802 41

\endlabellist

\begin{center}
\includegraphics[width=13.5cm]{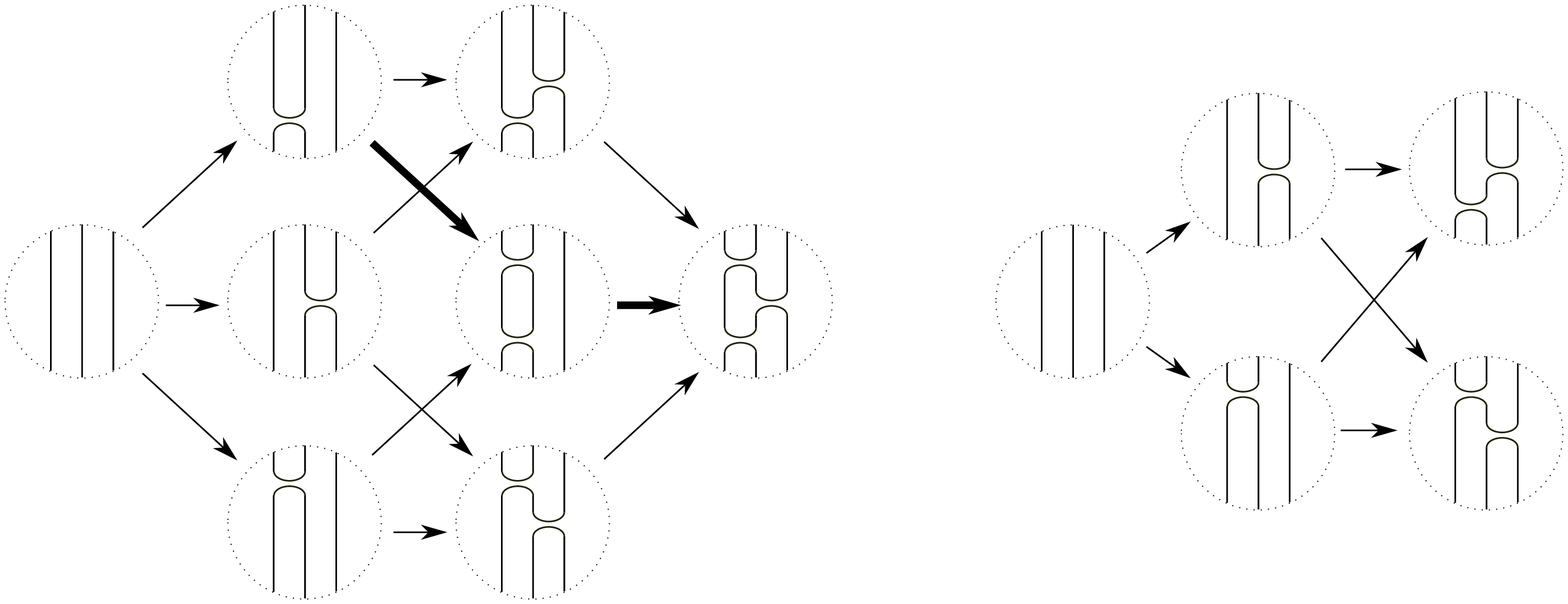}
\caption{\quad On the left is a 3 dimensional face of the hypercube. The tangle labeled $*ls$ represents the complex $(X^{(1)}_{*ls}, D^{(1)}_{*ls})$. The figure on the right is the result of cancellation among components of the maps represented by the thickened arrows on the left.}
\label{fig:web2a}
\end{center}
\end{figure}

\begin{figure}[!htbp]
\labellist 

\tiny 
\pinlabel $*s000$ at 40 112
\pinlabel $*s010$ at 160 112
\pinlabel $*s101$ at 286 112
\pinlabel $*s111$ at 407 112
\pinlabel $*s001$ at 160 232
\pinlabel $*s100$ at 160 -7
\pinlabel $*s011$ at 286 232
\pinlabel $*s110$ at 286 -7

\pinlabel $*s111$ at 808 110
\pinlabel $*s101$ at 707 181
\pinlabel $*s110$ at 707 39
\pinlabel $*s001$ at 583 181
\pinlabel $*s100$ at 583 39

\endlabellist
\begin{center}
\includegraphics[width=13.5cm]{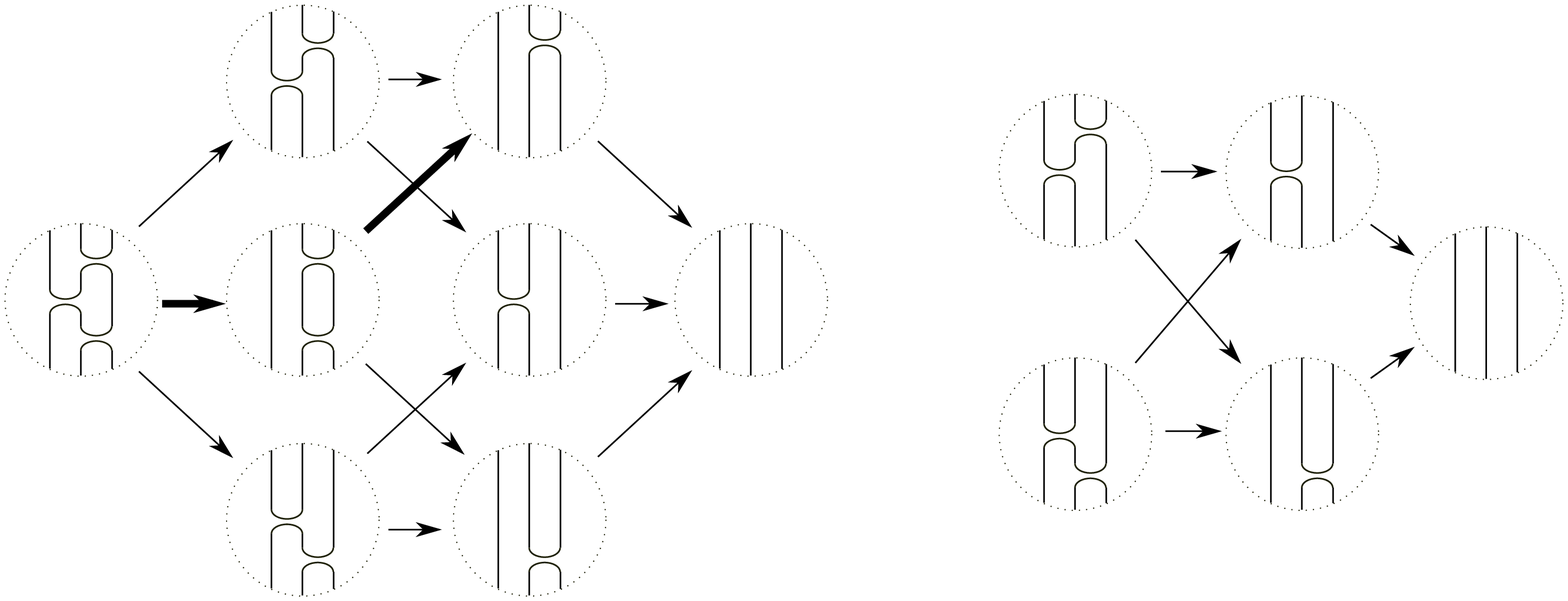}
\caption{\quad As in Figure \ref{fig:web2a}, the tangle labeled $*sl$ represents the complex $(X^{(1)}_{*sl}, D^{(1)}_{*sl})$. The figure on the right results from cancellation as before.}
\label{fig:web3a}
\end{center}
\end{figure}

The thickened arrows in this diagram represent a composition \begin{displaymath} 
\xymatrix{
X^{(1)}_{*001s}  \ar[r]^{A}& X^{(1)}_{*101s} \ar[r]^{B}& X^{(1)}_{*111s}    }
\end{displaymath} of the sort described in Figure \ref{fig:AB} of the previous subsection; that is, $A$ is a sum of splitting maps while $B$ is a sum of merging maps. After canceling the components $A_-$ and $B_+$ of these maps, one obtains the diagram on the right of Figure \ref{fig:web2a}. This cancellation introduces new maps from $X^{(1)}_{*100s}$ to $X^{(1)}_{*011s}$ which, according to \cite{kh1}, are simply the maps on reduced Khovanov homology induced by the obvious link cobordisms. We perform these cancellations for each $s\in \{0,1\}^3$.

Now consider the 3 dimensional face whose 8 vertices are $(X^{(1)}_{*sl}, D^{(1)}_{*sl})$, where $s\in\{0,1\}^3$ is fixed and $l$ ranges over $\{0,1\}^3$, as depicted on the left of Figure \ref{fig:web3a} (for every $s$, this face corresponds to the cube of resolutions for the tangle $y^{-1}x^{-1}y^{-1}$). Again, the thickened arrows represent a composition $B\circ A$ of splitting maps followed by merging maps. After canceling the components $A_-$ and $B_+$, one obtains the diagram on the right of Figure \ref{fig:web3a}. As before, we perform these cancellations for each $s\in \{0,1\}^3$.

\begin{figure}[!htbp]
\labellist 

\tiny 
\pinlabel $*000001$ at 40 513
\pinlabel $*000101$ at 40 395

\pinlabel $*010001$ at 222 745
\pinlabel $*010100$ at 222 628
\pinlabel $*000101$ at 222 513
\pinlabel $*000110$ at 222 395
\pinlabel $*100001$ at 222 280
\pinlabel $*100100$ at 222 164

\pinlabel $*011001$ at 406 924
\pinlabel $*011100$ at 406 807
\pinlabel $*010101$ at 406 691
\pinlabel $*010110$ at 406 574
\pinlabel $*000111$ at 406 457
\pinlabel $*100101$ at 406 340
\pinlabel $*100110$ at 406 223
\pinlabel $*110001$ at 406 107
\pinlabel $*110100$ at 406 -9

\pinlabel $*011101$ at 589 745
\pinlabel $*011110$ at 589 628
\pinlabel $*010111$ at 589 513
\pinlabel $*100111$ at 589 396
\pinlabel $*110101$ at 588 280
\pinlabel $*110110$ at 588 164

\pinlabel $*011111$ at 768 513
\pinlabel $*110111$ at 768 393


\endlabellist
\begin{center}
\includegraphics[width=13cm]{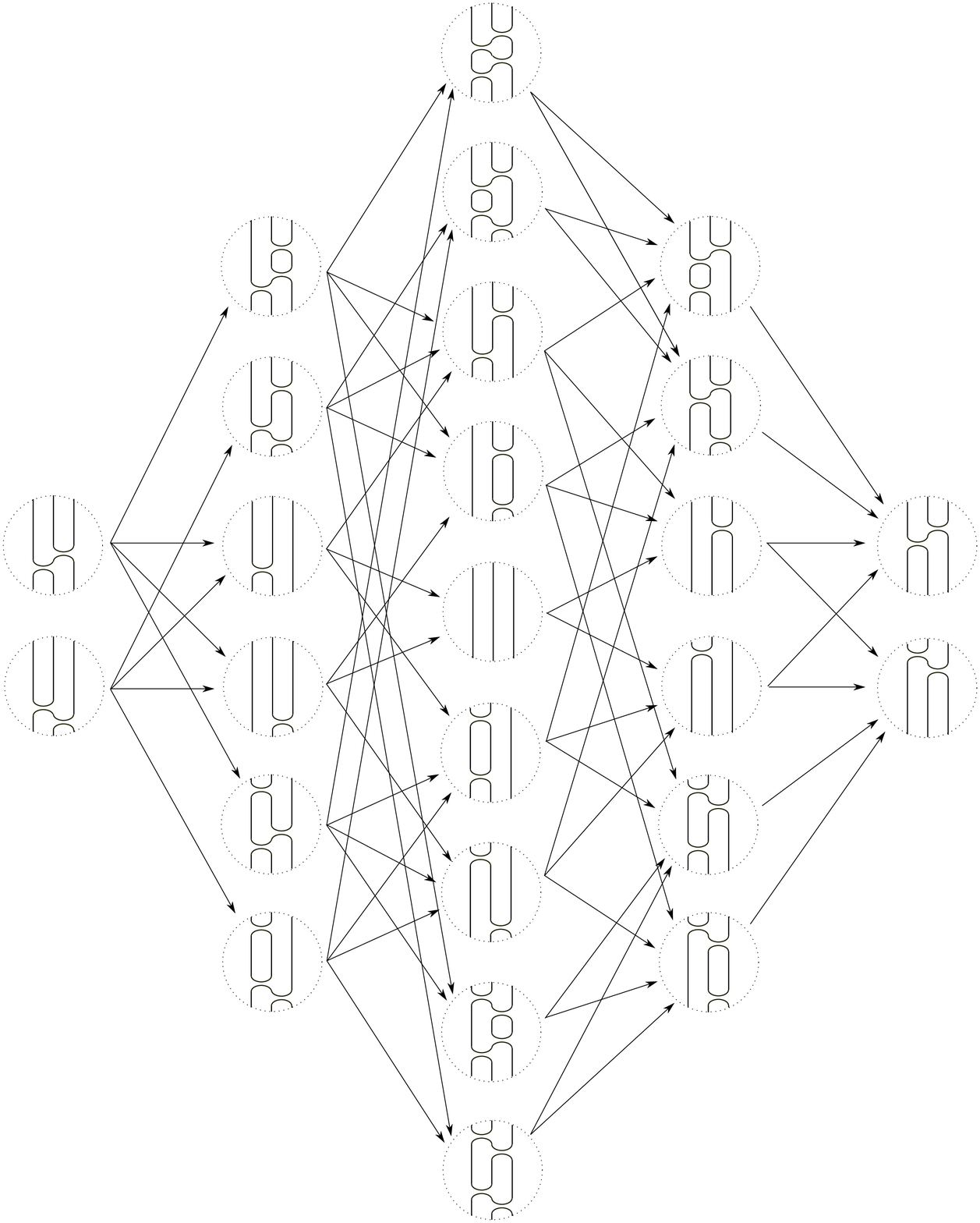}
\caption{\quad This is what remains of $(X^{(1)}, D^{(1)})$ after the preliminary cancellations. The tangle labeled $*j$ represents the complex $(X^{(1)}_{*j}, D^{(1)}_{*j})$.}
\label{fig:web11}
\end{center}
\end{figure}

\begin{figure}[!htbp]

\begin{center}
\includegraphics[width=15cm]{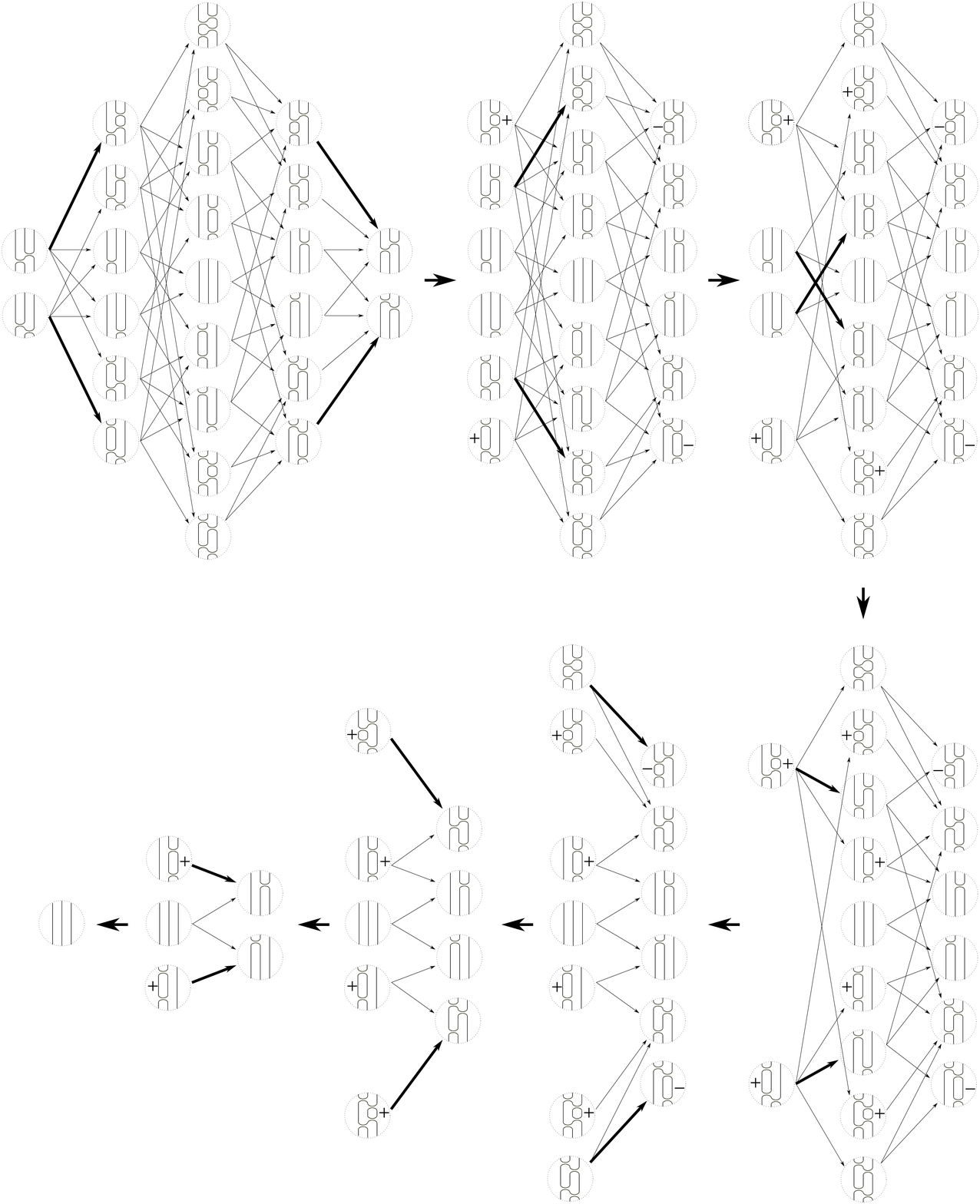}
\caption{\quad A sequence of cancellations which, in the end, leaves only the complex $(X^{(1)}_{*000111}, D^{(1)}_{*000111}).$ At each stage, the cancellations are performed as prescribed by the thickened arrows. A tangle containing a circle marked by a ``$+$" or ``$-$" represents the vector subspace of the corresponding $X^{(1)}_{*j}$ in which $\bv_+$ or $\bv_-$, respectively, is assigned to the circle. }
\label{fig:web16}
\end{center}
\end{figure}

Figure \ref{fig:web11} depicts what is left of the original 64 vertex hypercube for $(X^{(1)}, D^{(1)})$ after performing these preliminary cancellations. The arrows in Figure \ref{fig:web11} represent the components of the resulting maps between pairs, $X^{(1)}_{*j}$ and $X^{(1)}_{*j'}$ with $j<j'$, which increase the $h$-grading by 1 (these are just the maps indicated in the diagrams on the right sides of Figures \ref{fig:web2a} and \ref{fig:web3a}). Next, we perform the sequence of cancellations shown in Figure \ref{fig:web16}. At each stage in this sequence, the thickened arrows dictate which cancellations are to be made, as follows.

Suppose there is a thickened arrow in Figure \ref{fig:web16} from the tangle representing the vector space $Y$ to the tangle representing the vector space $Y'$. This arrow corresponds either to a sum of splitting maps from $Y$ to $Y'$, which we will denote by $A$, or to a sum of merging maps, which we will denote by $B$. In the first case, let $Y'_+$ (resp. $Y'_-$) be the vector subspace of $Y'$ generated by those elements for which $\bv_+$ (resp. $\bv_-$) is assigned to the circle that is split off (either of these subspaces may be empty). Then $$Y' \cong Y'_+ \oplus Y'_-,$$ and the component $A_-$ of $A$ which maps to the second summand in this decomposition is an isomorphism, as discussed in the previous subsection. The thickened arrow, in this case, indicates that we are to cancel this component $A_-$. $Y'_+$ is all that remains of $Y'$ after this cancellation. If $Y'_+$ is non-empty, then, after the cancellation, we represent this vector space by decorating the split-off circle in the corresponding tangle by a ``$+$" sign.

When the thickened arrow corresponds to a sum of merging maps, we let $Y_+$ (resp. $Y_-$) be the vector subspace of $Y$ generated by those elements for which $\bv_+$ (resp. $\bv_-$) is assigned to the circle that is merged. Again, $$Y \cong Y_+ \oplus Y_-,$$ and the restriction $B_+$ of $B$ to the first summand in this decomposition is an isomorphism. In this case, the thickened arrow indicates that we are to cancel this component $B_+$. After this cancellation, all that remains of $Y$ is the subspace $Y_-$. If $Y_-$ is non-empty, then, after the cancellation, we represent this vector space by decorating the merged circle in the corresponding tangle by a ``$-$" sign.

After performing the sequence of cancellations illustrated in Figure \ref{fig:web16}, all that remains is the complex $(X^{(1)}_{*000111}, D^{(1)}_{*000111})$. It follows that $(X^{(k)}, D^{(k)}) = (X^{(k)}_{*000111}, D^{(k)}_{*000111})$ for all $k \geq 2$. In particular, $E^k(\wt L) \cong E^k(L)$ as graded vector spaces for all $k\geq 2$. 

This completes the proof of Theorem \ref{thm:invcE}.

 \qed

\section{A transverse link invariant in $E^k(L)$}
\label{sec:transv}
Let $\xi_{rot} = \text{ker}(dz-ydx+xdy)$ be the rotationally symmetric tight contact structure on $S^3$. By a theorem of Bennequin \cite{benn}, any transverse link in $(S^3,\xi_{rot})$ is transversely isotopic to a closed braid around the $z$-axis. Conversely, it is clear that a closed braid around the $z$-axis may be isotoped through closed braids so that it becomes transverse (the contact planes are nearly vertical far enough from the $z$-axis). 

\begin{theorem}[{\rm \cite{osh,wrinkle}}] 
\label{thm:markov}If $L_1$ and $L_2$ are two closed braid diagrams which represent transversely isotopic links, then $L_2$ may be obtained from $L_1$ by a sequence of braid isotopies and positive braid stabilizations.
\end{theorem}

For a closed braid diagram $L$, Plamenevskaya defines a cycle $\wt\khc(L) \in (\ckh(L),d)$ whose image $\khc(L)$ in $\kh(L)$ is an invariant of the transverse link represented by $L$ \cite{pla1}. The cycle $\wt\khc(L)$ lives in the summand $\ckh(L_{I^o})$, where $I^o \in \{0,1\}^n$ is the vector which assigns a $0$ to every positive crossing and a $1$ to every negative crossing. In particular, $L_{I^o}$ is the oriented resolution of $L$, and the branched cover $-\Sigma(L_{I^o})$ is isomorphic to $\#^{s-1}(S^1 \times S^2)$, where $s$ is the number of strands in $L$. It is straightforward to check that, under the identification of $\ckh(L)$ with $$E^1(L) \cong \bigoplus_{I \in \{0,1\}^m} \hf(-\Sigma(L_I)),$$ the cycle $\wt \khc(L)$ is identified with the element $\khc^1(L)$ with the lowest quantum grading in the summand $\hf(-\Sigma(L_{I^o}))$ (compare the definition of $\wt \khc(L)$ in \cite{pla1} with the description of $\ckh(L) \cong E^1(L)$ in \cite[Sections 5 and 6]{osz12}). In this section, we show that $\khc^1(L)$ gives rise to an element $\khc^k(L) \in E^k(L)$ for every $k>1$. The proposition below makes this precise.

\begin{proposition}
\label{prop:cycle}
The element $\khc^k(L)$, defined recursively by $$\khc^k(L) = [\khc^{(k-1)}(L)] \,\,\in \,\,H_*(E^{k-1}(L),D^{k-1}) = E^k(L),$$ is a cycle in $(E^k(L),D^k)$ for every $k > 1$. 
\end{proposition} 

Note that Plamenevskaya's invariant $\khc(L)$ is identified with $\khc^2(L)$ under the isomorphism between $\kh(L)$ and $E^2(L)$.

\begin{proof}[Proof of Proposition \ref{prop:cycle}]
First, we consider the case in which $L$ has an odd number of strands. In this case, the braid axis of $L$ lifts to a fibered knot $B\subset -\Sigma(L)$. In \cite{lrob1}, Roberts observes that $B$ gives rise to \emph{another} grading of the complex $(X,D)$ associated to $L$; we refer to this as the ``$A$-grading" of $(X,D)$. The $A$-grading gives rise to an ``$A$-filtration" of $(X,D)$, and Roberts shows that $\khc^1(L)$ is the unique element of $(X^{(1)},D^{(1)})$ in the bottommost $A$-filtration level (see also \cite{baldpla}). Since $D^{(1)}$ does not increase $A$-filtration level (as $D$ is an $A$-filtered map), it follows that the element $\khc^k(L)$ defined in Proposition \ref{prop:cycle} is a cycle in $(X^{(k)},D^{(k)})$ and, hence, in $(E^k(L), D^k)$ for every $k\geq 1$.

Now, suppose that $L$ has an even number of strands, and let $L^+$ be the diagram obtained from $L$ via a positive braid stabilization (i.e. a positive Reidemeister I move). For $k \geq 1$, let $$\rho^k:(E^k(L), D^k) \rightarrow (E^k(L^+), D^k)$$ be the chain map induced by the map $F_{\infty,0}\oplus F_{\infty,1}$ defined in Subsection \ref{ssec:RI}. Recall that $\rho^1=A$ is the sum of the maps $$A_I:\hf(-\Sigma(L^+_{I \times \{\infty\}})) \rightarrow \hf(-\Sigma(L^+_{I \times \{0\}}))$$ over all $I \in \{0,1\}^m$. Let $I^o \in \{0,1\}^{m+1}$ be the vector for which $L^+_{I^o}$ is the oriented resolution of $L$, and define $\wt I^o \in \{0,1\}^m$ by $\wt I^o \times \{0\}=I^o$. Then $\khc^1(L)$ is the element with the lowest quantum grading in $\hf(-\Sigma(L^+_{\wt I^o \times \{\infty\}})),$ and $\khc^1(L^+)$ is the element with the lowest quantum grading in $\hf(-\Sigma(L^+_{\wt I^o \times \{0\}}))$. Since $A_{\wt I^o}$ is the map induced by the 2-handle cobordism from $-\Sigma(L^+_{\wt  I^o \times \{\infty\}})$ to $-\Sigma(L^+_{\wt  I^o \times \{0\}})$ corresponding to $0$-surgery on an unknot, $A_{\wt  I^o}$ sends $\khc^1(L)$ to $\khc^1(L^+)$ (see the discussion of gradings in \cite{osz6}).

Proposition \ref{prop:cycle} now follows by induction. Indeed, suppose that $\rho^{k-1}$ sends $\khc^{(k-1)}(L)$ to $\khc^{(k-1)}(L^+)$ for some $k>1$. Then, since $\khc^{(k-1)}(L^+)$ is a cycle in $(E^{k-1}(L^+),D^{k-1})$ (as $L^+$ has an odd number of strands) and $\rho^{k-1}$ is injective (in fact, $\rho^k$ is an isomorphism for $k\geq 2$), it follows that $\khc^{(k-1)}(L)$ is a cycle in $(E^{k-1}(L),D^{k-1})$, and that $\rho^k$ sends $\khc^{k}(L)$ to $\khc^k(L^+)$.  
\end{proof}

According to the proposition below, the element $\khc^k(L) \in E^k(L)$ is an invariant of the transverse link in $(S^3,\xi_{rot})$ represented by $L$ for each $k \geq 2$.

\begin{proposition}
\label{prop:trans}
If the closed braid diagrams $L_1$ and $L_2$ represent transversely isotopic links in $(S^3,\xi_{rot}),$ then there is an isomorphism from $E^k(L_1)$ to $E^k(L_2)$ which preserves $h$-grading and sends $\khc^k(L_1)$ to $\khc^k(L_2)$ for each $k\geq 2$.
\end{proposition}

\begin{proof}[Proof of Proposition \ref{prop:trans}]
According to Theorem \ref{thm:markov}, it suffices to check Proposition \ref{prop:trans} for diagrams which differ by a positive braid stabilization or a braid isotopy. If $L^+$ is the diagram obtained from $L$ via a positive braid stabilization, then the isomorphism $$\rho^k: E^k(L) \rightarrow E^k(L^+)$$ sends $\khc^k(L)$ to $\khc^k(L^+)$ for each $k\geq 2$, as shown in the proof of Proposition \ref{prop:cycle}. 

Every braid isotopy is a composition of Reidemeister II and III moves. Suppose that $\wt L$ is the diagram obtained from $L$ via a Reidemeister II move. In this case, $(X^{(k)}, D^{(k)}) \cong (X_{*01}^{(k)}, D_{*01}^{(k)})$ for each $k \geq 2$, where $(X,D)$ and $(X_{*01},  D_{*01})$ are the complexes associated to $\wt L$ and $L$, respectively (see Subsection \ref{ssec:RII}). Under this isomorphism, $\khc^k(L^+)$ is clearly identified with $\khc^k(L)$. 

The same sort of argument applies when $\wt L$ is the diagram obtained from $L$ by replacing a trivial 3-tangle with the tangle associated to the braid word $xyxy^{-1}x^{-1}y^{-1}.$ In this case, $(X^{(k)}, D^{(k)}) = (X_{*000111}^{(k)}, D_{*000111}^{(k)})$ for each $k \geq 2$, where $(X,D)$ and $(X_{*000111}, D_{*000111})$ are the complexes associated to $\wt L$ and $L$ (see Subsection \ref{ssec:RIII}). Again, it is clear that $\khc^k(L^+)$ is identified with $\khc^k(L)$ under this isomorphism.

\end{proof}

The proof of Proposition \ref{prop:vanishing} follows along the same lines as the proof of Proposition 1.4 in \cite{baldpla}. We may assume that the braid diagram $L$ for our transverse link has $2s+1$ strands. The complex $(X,D)$ associated to the diagram $L$ is generated by elements which are homogeneous with respect to both the $h$-grading and the $A$-grading mentioned in the proof of Proposition \ref{prop:cycle}. After canceling all components of $D$ which do not shift either of the $h$- or $A$-gradings, we obtain a complex $(X',D')$ which is bi-filtered chain homotopy equivalent to $(X,D)$. Let $E^k(L)'$ denote the $E^k$ term of the spectral sequence associated to the $h$-filtration of $(X',D')$ (clearly, $E^k(L)'$ is isomorphic to $E^k(L)$). 

Roberts shows that there is a unique element $c \in (X',D')$ in $A$-filtration level $-s$, whose image in $H_*(X',D') \cong \hf(-\Sigma(L))$ corresponds to the contact element $c(\xi_L)$, and whose image in $E^1(L)'$ corresponds to $\khc^1(L)$. Therefore, Proposition \ref{prop:vanishing} boils down to the statement that if the image of $c$ in $E^k(L)'$ vanishes and $E^k(L)'$ is supported in non-positive $h$-gradings, then the image of $c$ in $H_*(X',D')$ vanishes.

\begin{proof}[Proof of Proposition \ref{prop:vanishing}]
We will prove this by induction on $k$. Suppose that the statement above holds for $1\leq n<k$ (it holds vacuously for $n=1$). Let $c^{(k-1)}$ denote the element of $X'^{(k-1)}$ represented by $c$, and assume that $c^{(k-1)}$ is non-zero. Then the image of $c$ in $H_*(X',D')$ corresponds to the image of $c^{(k-1)}$ in $H_*(X'^{(k-1)}, D'^{(k-1)})$.

Let $K = n(L)-n_-(L)$, where $n(L)$ is the total number of crossings in $L$. The $h$-filtration of $(X',D')$ induces an $h$-fitration of $(X'^{(k-1)},D'^{(k-1)})$: $$\{0\} = \mathcal{F}_{K+1} \subset  \mathcal{F}_{K} \subset \dots \subset  \mathcal{F}_{-n_-(L)} = X'^{(k-1)}.$$ Let us assume that $E^k(L)'$ is supported in non-positive $h$-gradings. If the image of $c$ in $E^k(L)'$ is zero, then there must exist some $y \in X'^{(k-1)}$ with $h(y)=-(k-1)$ such that $D'^{(k-1)}(y) = c^{(k-1)} + x$, where $x \in \mathcal{F}_1.$ Let $J$ be the greatest integer for which there exists some $y'$ such that $D'^{(k-1)}(y') = c^{(k-1)} + x'$, where $x' \in \mathcal{F}_J$. We will show that $J=K+1$, which implies that $x'=0$, and, hence, that $c^{(k-1)}$ is a boundary in $(X'^{(k-1)}, D'^{(k-1)})$ (which implies that $c$ is a boundary in $(X',D')$).

Suppose, for a contradiction, that $J<K+1$. Write $x' = x_J + x''$, where $h(x_J) = J$ and $x'' \in \mathcal{F}_{J+1}$. Note that $D'^{(k-1)}(x_J+x'')=0$ as $x' = x_J + x''$ is homologous to $c^{(k-1)}$. Since every component of $D'^{(k-1)}$ shifts the $h$-grading by at least $k-1$, it follows that $D'^{(k-1)}(x'') \in \mathcal{F}_{J+k}.$ But this implies that $D'^{(k-1)}(x_J) \in  \mathcal{F}_{J+k}$ as well, since $D'^{(k-1)}(x_J + x'') = 0$. Therefore, $x_J$ represents a cycle in $(E^{k-1}(L)',D'^{k-1})$. Since $J\geq 1$ and $E^k(L)'$ is supported in non-positive $h$-gradings, it must be that $x_J$ is also a \emph{boundary} in $(E^{k-1}(L)',D'^{k-1})$. That is, there is some $y''$ with $h(y'') = J-(k-1)$ such that $D'^{(k-1)}(y'') = x_J + x'''$, where $x''' \in \mathcal{F}_{J+1}$. But then, $D'^{(k-1)}(y'+y'') = c^{(k-1)}+(x'' + x''')$, and the fact that $x'' + x'''$ is contained in $\mathcal{F}_{J+1}$ contradicts our earlier assumption on the maximality of $J$.

To finish the proof of Proposition \ref{prop:vanishing}, recall that $c(\xi_L)=0$ implies that $\xi_L$ is not strongly symplectically fillable \cite{osz2}.

\end{proof}


\section{Two examples}
\label{sec:example}

Recall that a planar link diagram is said to be \emph{almost alternating} if one crossing change makes it alternating. An almost alternating link is a link with an almost alternating diagram, but no alternating diagram. As was mentioned in the introduction, Conjecture \ref{conj:quant} holds for almost alternating links.

\begin{lemma}
\label{lem:almostalt}
If the link $L$ is almost alternating, then there is a well-defined quantum grading on $E^k(L)$ for $k\geq 2$, and $D^k$ increases this grading by $2k-2$.
\end{lemma}

\begin{proof}[Proof of Lemma \ref{lem:almostalt}]
In an abuse of notation, we let $L$ denote both the underlying link and an almost alternating diagram for the link. The diagram $L$ can be made alternating by changing some crossing $c$. Let $L_0$ and $L_1$ be the link diagrams obtained by taking the $0$- and $1$-resolutions of $L$ at $c$. Since both $L_0$ and $L_1$ are alternating, each of $\kh(L_0)$ and $\kh(L_1)$ is supported in a single $\delta$-grading. If $G$ is a bi-graded group, we let $G[a,b]$ denote the group obtained from $G$ by shifting the bi-grading by $[a,b]$. It follows from Khovanov's original definition \cite{kh1} that $\kh(L)$ is the homology of the mapping cone of a map $$f:\kh(L_0)[a,b] \rightarrow \kh(L_1)[c,d]$$ for some integers $a$, $b$, $c$ and $d$. Therefore, $\kh(L)$ is supported in at most two $\delta$-gradings.

If $\kh(L)$ is supported in one $\delta$-grading then the lemma holds trivially since all higher differentials vanish and $E^k(L) \cong \kh(L)$ for each $k\geq 2$. Otherwise, $\kh(L)$ is supported in two consecutive $\delta$-gradings, $\delta_0$ and $\delta_1 = \delta_0-1$, corresponding to the two groups $\kh(L_0)[a,b]$ and $\kh(L_1)[c,d]$, respectively \cite{AP,ck2}. Since all higher differentials vanish in the spectral sequences associated to $L_0$ and $L_1$, any non-trivial higher differential in the spectral sequence associated to $L$ must shift the $\delta$ grading by $-1$. And, because the differential $D^k$ shifts the homological grading by $k$, it follows that $D^k$ shifts the quantum grading by $2k-2$ (recall that the $\delta$-grading is one-half the quantum grading minus the homological grading). Well-definedness of the quantum grading on each term follows from the fact that the differentials $D^k$ are homogeneous.

\end{proof}

We now use Lemma \ref{lem:almostalt} together with Proposition \ref{prop:cycle} to compute the higher terms in the spectral sequence associated to the torus knot $T(3,5)$. First, note that the branched double cover $\Sigma(T(3,5))$ is the Poincar{\'e} homology sphere $\Sigma(2,3,5)$, whose Heegaard Floer homology has rank 1. On the other hand, the reduced Khovanov homology of $T(3,5)$ has rank 7, and its Poincar{\'e} polynomial is $$\kh(T(3,5))(h,q) = \underline{h^{0}q^{8}} + \underline{h^{2}q^{12}} + \underline{h^{3}q^{14}} + h^{4}q^{14} + \underline{h^{5}q^{18}} + h^{6}q^{18} + h^{7}q^{20}$$ (here, the exponent of $h$ indicates the homological grading, while the exponent of $q$ indicates the quantum grading). The underlined terms represent generators supported in $\delta$-grading 4, while the other terms represent generators supported in $\delta$-grading 3. The grid in Figure \ref{fig:grid} depicts $\kh(T(3,5))$; the numbers on the horizontal and vertical axes are the homological and quantum gradings, respectively, and a dot on the grid represents a generator in the corresponding bi-grading.

 \begin{figure}[!htbp]

 \labellist 
 \small\pinlabel $q$ at 23 351 
\tiny\hair 2pt

\pinlabel $8$ at 6 43 
\pinlabel $10$ at 6 84 
\pinlabel $12$ at 6 128 
\pinlabel $14$ at 6 170 
\pinlabel $16$ at 6 211 
\pinlabel $18$ at 6 252 
\pinlabel $20$ at 6 293 

\pinlabel $0$ at 46 5 
\pinlabel $1$ at 88 5
\pinlabel $2$ at 130 5
\pinlabel $3$ at 172 5
\pinlabel $4$ at 215 5
\pinlabel $5$ at 259 5
\pinlabel $6$ at 300 5
\pinlabel $7$ at 343 5

\small\pinlabel $h$ at 396 22

\pinlabel $\bullet$ at 46 43
\pinlabel $\bullet$ at 130 128
\pinlabel $\bullet$ at 172 170
\pinlabel $\bullet$ at 215 170
\pinlabel $\bullet$ at 259 252
\pinlabel $\bullet$ at 300 252
\pinlabel $\bullet$ at 343 293

\endlabellist
 \begin{center}
\includegraphics[height=5.7cm]{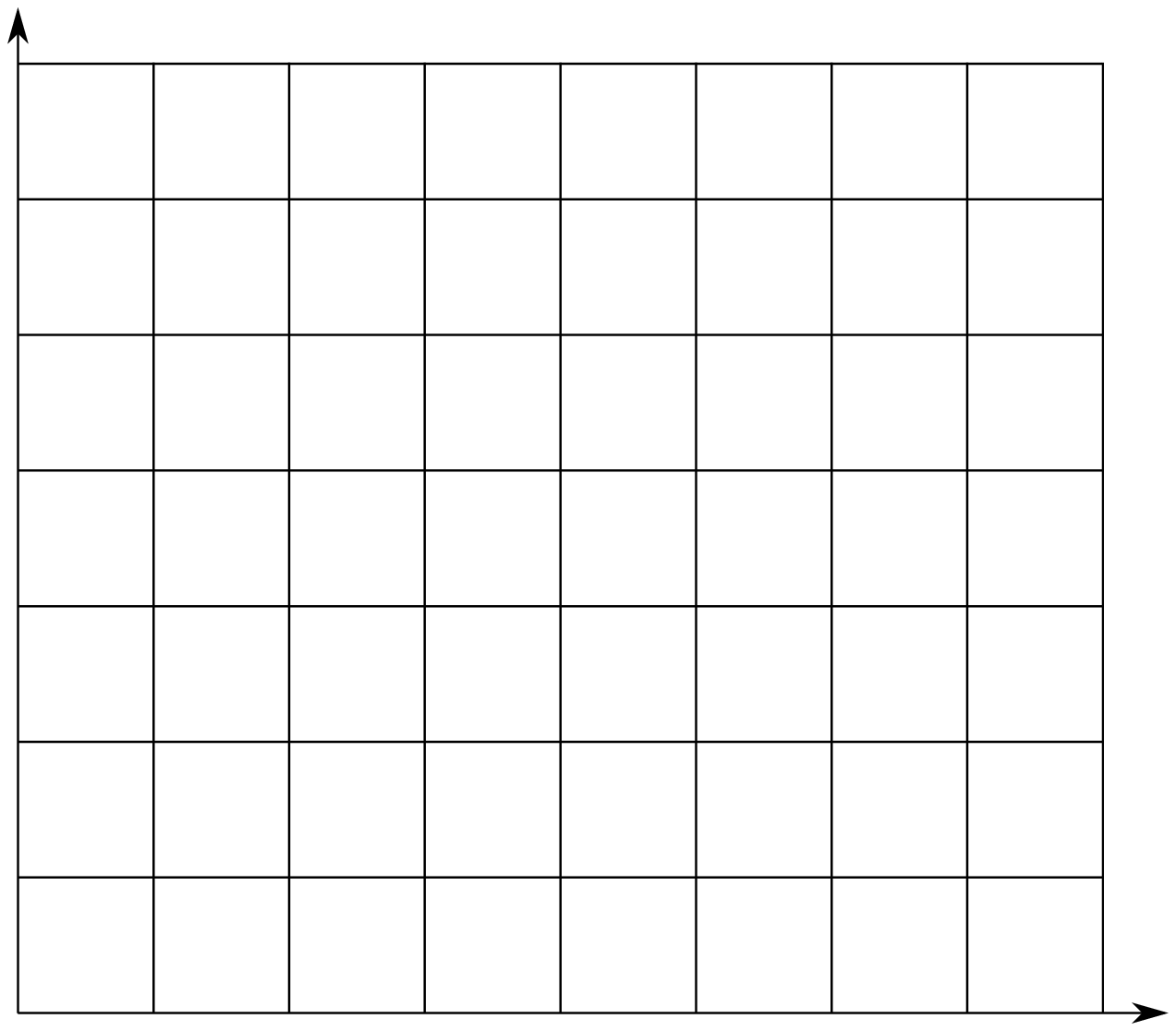}
\caption{\quad $\kh(T(3,5))$.}
\label{fig:grid}
\end{center}
\end{figure}

Observe that $T(3,5)$ is the same as the pretzel knot $P(-2,3,5)$ (both describe the knot $10_{124}$). Therefore, $T(3,5)$ is almost alternating, as indicated in Figure \ref{fig:almostalt}, and the differential $D^k$ increases quantum grading by $2k-2$ by Lemma \ref{lem:almostalt}. Moreover, if $x$ and $y$ are the elementary generators of the braid group on three strands, then $T(3,5)$ is the closure of the 3-braid specified by the word $(xy)^5$. If $\mathcal{T}$ is the transverse knot in $S^3$ represented by this closed braid, then $\psi(\mathcal{T}) \neq 0$ since the braid is positive \cite{pla1}. 

\begin{figure}[!h]
\label{fig:almostalt}

\begin{center}
\includegraphics[width=7.5cm]{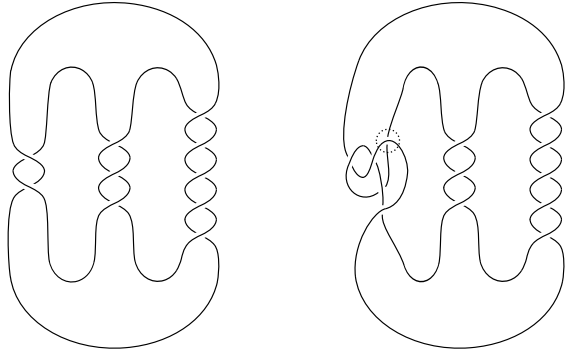}
\caption{\quad On the left is a standard diagram for $T(3,5) = P(-2,3,5).$ On the right is an almost alternating diagram for $P(-2,3,5)$ which can be made alternating by changing the circled crossing. }
\end{center}
\end{figure}

Note that $\psi(\mathcal{T})$ is the generator of $\kh(T(3,5))$ in bi-grading $(0,8)$. By Proposition \ref{prop:cycle}, $\psi(\mathcal{T})$ gives rise to a cycle $\psi^k(\mathcal{T})$ in $(E^k(T(3,5)), D^k)$ for $k\geq 2$. Since $\kh(T(3,5))$ is supported in non-negative homological gradings and the differential $D^k$ increases homological grading by $k$, each $\psi^k(\mathcal{T})$ is non-zero. Therefore, $\psi^{\infty}(\mathcal{T})$ generates $E^{\infty}(T(3,5)) \cong \hf(-\Sigma(T(3,5)) \cong \zzt$. 

Moreover, since $D^k$ increases the bi-grading by $(k,2k-2)$ and the generators in positive homological gradings must eventually die, there is only one possibility for the higher differentials, as indicated in Figure \ref{fig:grid5}. Therefore, the Poincar{\'e} polynomials for the bi-graded groups $E^2(T(3,5))$, $E^3(T(3,5))$ and $E^4(T(3,5)) = E^{\infty}(T(3,5))$ are 

\begin{eqnarray*}
E^2(T(3,5))(h,q)&=& h^{0}q^{8} + h^{2}q^{12} + h^{3}q^{14} + h^{4}q^{14} + h^{5}q^{18} + h^{6}q^{18} + h^{7}q^{20},\\ 
E^3(T(3,5))(h,q) &=& h^{0}q^{8}  + h^{3}q^{14}   + h^{6}q^{18},\\
E^4(T(3,5))(h,q) &=& h^{0}q^{8}. 
\end{eqnarray*}
\vspace{1mm}

 \begin{figure}[!htbp]

 \labellist 
 \small\pinlabel $q$ at 23 351 
\tiny\hair 2pt

\pinlabel $8$ at 6 43 
\pinlabel $10$ at 6 84 
\pinlabel $12$ at 6 128 
\pinlabel $14$ at 6 170 
\pinlabel $16$ at 6 211 
\pinlabel $18$ at 6 252 
\pinlabel $20$ at 6 293 

\pinlabel $0$ at 46 5 
\pinlabel $1$ at 88 5
\pinlabel $2$ at 130 5
\pinlabel $3$ at 172 5
\pinlabel $4$ at 215 5
\pinlabel $5$ at 259 5
\pinlabel $6$ at 300 5
\pinlabel $7$ at 343 5

\small\pinlabel $D^2$ at 172 133
\pinlabel $D^3$ at 215 211
\pinlabel $D^2$ at 300 289

\small\pinlabel $h$ at 396 22

\pinlabel $\bullet$ at 46 43
\pinlabel $\bullet$ at 130 128
\pinlabel $\bullet$ at 174 170
\pinlabel $\bullet$ at 217 169
\pinlabel $\bullet$ at 259 253
\pinlabel $\bullet$ at 300 253
\pinlabel $\bullet$ at 344 294

\endlabellist
 \begin{center}
\includegraphics[height=5.7cm]{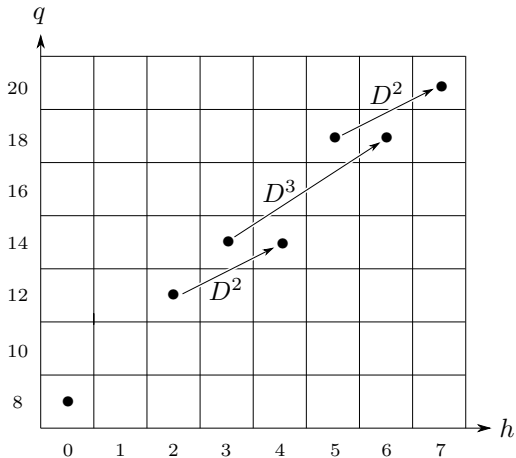}
\caption{\quad The higher differentials.}
\label{fig:grid5}
\end{center}
\end{figure}

We can apply the same sort of reasoning in studying the torus knot $T(3,4),$ as it too is almost alternating (see Figure \ref{fig:t34}). The reduced Khovanov homology of $T(3,4)$ has rank 5, and its Poincar{\'e} polynomial is $$\kh(T(3,4))(h,q) = \underline{h^{0}q^{6}} + \underline{h^{2}q^{10}} + \underline{h^{3}q^{12}} + h^{4}q^{12} + \underline{h^{5}q^{16}}
.$$ The underlined terms represent generators supported in $\delta$-grading 3, while the term $h^{4}q^{12}$ represents the generator in $\delta$-grading 2. Meanwhile, $\text{rk}\,\hf(-\Sigma(T(3,4))) = 3$. 

Since $T(3,4)$ is the closure of the positive 3-braid specified by $(xy)^4$, the generator represented by $h^0q^6$ must survive as a non-zero cycle throughout the spectral sequence. With this restriction, there is only one possibility for the higher differentials; namely, there is a single non-trivial component of the $D^2$ differential sending the generator represented by $h^{2}q^{10}$ to the generator represented by $h^{4} q^{12}$, and the spectral sequence collapses at the $E^3$ term. That is, the Poincar{\'e} polynomials for the bi-graded groups $E^2(T(3,4))$ and $E^3(T(3,4))= E^{\infty}(T(3,4))$ are 
\begin{eqnarray*}
E^2(T(3,4))(h,q)&=&h^{0}q^{6} + h^{2}q^{10} + h^{3}q^{12} + h^{4}q^{12} + h^{5}q^{16},\\ 
E^3(T(3,4))(h,q) &=& h^{0}q^{6}  + h^{3}q^{12} + h^{5}q^{16}.\\
\end{eqnarray*}

\begin{figure}[!h]
\label{fig:t34}

\begin{center}
\includegraphics[width=3.4cm]{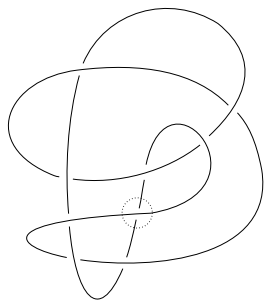}
\caption{\quad An almost alternating diagram for the knot $T(3,4) = 8_{19}$ which can be made alternating by changing the circled crossing. }
\end{center}
\end{figure} 

\begin{remark}
By taking connected sums and applying Corollary \ref{cor:connect}, we can use these examples to produce infinitely many links all of whose $E^k$ terms we know and whose spectral sequences do not collapse at $E^2$.
\end{remark}

\bibliographystyle{hplain}
\bibliography{References}

\end{document}